\documentclass{ws-m3as}

\usepackage{amsmath}
\usepackage{amssymb}
\usepackage{t1enc}
\usepackage[english]{babel}
\usepackage{verbatim}
\usepackage{graphicx}
\usepackage[dvips]{epsfig}
\usepackage{url}
\usepackage{tikz}
\usetikzlibrary{arrows.meta}
\usepackage{subcaption}
\usepackage{booktabs}
\DeclareFontFamily{U}{mathx}{}
\DeclareFontShape{U}{mathx}{m}{n}{<-> mathx10}{}
\DeclareSymbolFont{mathx}{U}{mathx}{m}{n}
\DeclareMathAccent{\widecheck}{0}{mathx}{"71}

\usepackage[hidelinks]{hyperref}
\usepackage{pgfplots}

\newcommand{\closure}{\mathrm{cl}\,}
\newcommand{\xx}{\mathbf{x}}
\newcommand{\ul}{\underline}
\newcommand{\wh}{\widehat}

\newenvironment{proofof}[1]{%
\par\addvspace{12pt plus3pt minus3pt}\global\logotrue%
\noindent{\bf Proof of #1.\hskip.5em}\ignorespaces}{%
	\par\iflogo\vskip-\lastskip
	\vskip-\baselineskip\prbox\par
	\addvspace{12pt plus3pt minus3pt}\fi}

\usepackage{array}

\usepackage{graphicx,bm,xcolor}
\usetikzlibrary{arrows}
\newtheorem{notation}{Notation}[section]

\begin{document}

\markboth{Stefan Takacs}{An IETI method for conforming discretizations of the biharmonic problem}

\catchline{}{}{}{}{}

\title{An Isogeometric Tearing and Interconnecting method for conforming discretizations of the biharmonic problem}

\author{Stefan Takacs}

\address{Institute of Numerical Mathematics,\\
Johannes Kepler University Linz,\\
Altenberger Str.~69, 4040 Linz, Austria\\
stefan.takacs@numa.uni-linz.ac.at}

\maketitle
$ $\\[2em]

\selectlanguage{english}
\begin{abstract}
	We propose and analyze a domain decomposition solver for
	the biharmonic problem. The problem is discretized in a conforming way
	using multi-patch Isogeometric Analysis.
	As first step, we discuss the setup of a sufficiently smooth discretization space.
	We focus on two dimensional computational domains
	that are parameterized with sufficiently smooth geometry functions.
	As solution technique, we use a variant of the Dual-Primal Finite Element
	Tearing and Interconnecting method that is also known as Dual-Primal
	Isogeometric Tearing and Interconnecting (IETI-DP) method in the context of
	Isogeometric Analysis.  We present a condition number estimate
	and illustrate the behavior of the proposed method with numerical results.
\end{abstract}

\keywords{Biharmonic problem, Isogeometric Analysis; FETI-DP.}

\ccode{AMS Subject Classification: 65N55, 65N30, 65F08, 65D07}

\section{Introduction}
\label{sec:1}

This paper is concerned with the conforming discretization and the solution of elliptic boundary value problems for partial differential equations (PDEs) of fourth order. Examples for applications of such differential equations include the simulation of thin plates or shells as in the Kirchhoff-Love model, cf. Refs.~\refcite{Ciarlet:2002} and~\refcite{Reddy:2007}, and for strong formulations of Schur complements, which were proposed as preconditioners for certain classes of optimal control problems, cf. Ref.~\refcite{MardalNielsenNordaas:2017} and references therein. We consider the first biharmonic equation as a model problem since a profound understanding of solvers for that equation also helps with the construction of solvers for other fourth order problems.

The variational formulation of the biharmonic problem uses the Sobolev space $H^2$ of twice weakly differentiable functions. In order to obtain a conforming discretization with piecewise polynomial functions (as usual in finite element and related approaches), the basis functions need to be continuously differentiable. In the context of finite element methods, this can be achieved with the Argyris element, cf. Ref.~\refcite{Ciarlet:2002}, that uses a polynomial degree of $5$, which leads to $21$ degrees of freedom per element.
Simpler constructions are possible using Isogeometric Analysis (IgA), which
was proposed in Ref.~\refcite{HughesCottrellBazilevs:2005} to bride the gap between computer aided design and simulation, also with applications in continuum mechanics in mind.
In IgA, both the computational domain and the solution of the PDE are represented as linear combination of tensor-product B-splines or non-uniform rational B-splines (NURBS).
So, arbitrarily smooth basis functions can be constructed effortlessly.

Alternatives to conforming discretizations include mortar or Nitsche type methods, cf. Refs.~\refcite{BrivadisBuffaWohlmuthWunderlich:2015,GuoRuess:2015,HorgerRealiWohlmuthWunderlich:2019,WeinmullerTakacs:2022} and~\refcite{BenvenutiLoliSangalliTakacs:2023}, or methods based on the reformulation of the problem in terms of a system of differential equations, like in the Hellan--Herrmann--Johnson approach, cf. Refs.~\refcite{Hellan:1967,Herrmann:1967} and~\refcite{Johnson:1973}. While allowing for simpler discretizations, these approaches come with their own sets of challenges. The construction of solvers for such modified problems is often harder than for the original problem. The construction of the Hellan--Herrmann--Johnson method is tied to the specific differential equation and the extension to variants of the problem, like with highly varying material parameters, is not obvious.

We are interested in a domain decomposition solver for the conforming isogeometric discretization the the biharmonic problem. Several approaches have already been proposed for the single-patch case, like multigrid methods, cf. Refs.~\refcite{SognTakacs:2019} and~\refcite{SognTakacs:2023}. Also the extension of the fast diagonalization preconditioner, cf. Ref.~\refcite{SangalliTani:2016}, to the biharmonic problem seems to be straight-forward.
The situation for the multi-patch case is more sparse.

Among the most popular domain decomposition methods are the
Finite Element Tearing and Interconnecting (FETI) method and its dual-primal version FETI-DP, see for example Refs.~\refcite{FarhatRoux:1991a,FarhatLesoinneLeTallecPiersonRixen:2001a} and~\refcite{ToselliWidlund:2005a}, and
the Balancing Domain Decomposition by Constraints (BDDC) method, introduced in Ref.~\refcite{Dohrmann:2003}, which can be seen as the primal counterpart of the FETI-DP method. The BDDC approach has been adapted to IgA in Refs.~\refcite{BdVPavarinoScacciWidlundZampini:2014} and~\refcite{BdvPavarinoScacciWidlundZampini:2017}; see also more recent publications of the authors for further extensions. In these papers, single-patch geometries are considered, which are subdivided into substructures used for the domain decomposition solver. In their construction, the authors keep the smoothness of the spline space.

Also FETI-DP has been adapted to IgA. Ref.~\refcite{KleissPechsteinJuttlerTomar:2012} extends that method to multi-patch IgA for second order problems. The authors have called their approach the Dual-Primal Isogeometric Tearing and Interconnecting method (IETI-DP). Compared to the abovementioned publications on BDDC, Ref.~\refcite{KleissPechsteinJuttlerTomar:2012} takes another approach to the setup of the substructures: Multi-patch geometries are considered and each patch from the geometry representation is treated as one substructure for the solver. The coupling between the patches is realized with minimal smoothness required for a conforming discretization, so it is only $C^0$ continuous for second order problems. Later publications, like Ref.~\refcite{SchneckenleitnerTakacs:2020}, which gave a condition number bound for the IETI-DP solvers, followed that setup.

We also follow that idea and assume each substructure to be one patch of a multi-patch discretization. Since we are interested in a $H^2$ conforming discretization, we enforce $C^1$ continuity across the interfaces between any two patches. This is not trivial for general geometries. We assume that at each interface between any two patches both their parametrizations and the normal derivatives of their parametrizations agree.
More general settings, like analysis-suitable $G^1$ smooth discretizations as proposed in Refs.~\refcite{CollinSangalliTakacs:2016,KaplSangalliTakacs:2019} and~\refcite{HughesSangalliTakacsToshniwal:2021},
would need some additional considerations which are outside of the scope of this paper.
In Ref.~\refcite{Kapl:2025}, IETI-DP solvers for such a setup were proposed and numerically tested, but so far without condition number estimates.
The analysis presented in the manuscript at hand can be seen as a first step towards a comprehensive analysis for such a more general setup.
We will use a basis transformation such that we can use a scaled Dirichlet preconditioner, which is cheaper than the deluxe preconditioners proposed in Ref.~\refcite{BdVPavarinoScacciWidlundZampini:2014} in the context of BDDC solvers.
For the analysis, we explicitly construct discrete biharmonic extensions with ideas similar to those of Refs.~\refcite{Nepomnyaschikh:1995} and~\refcite{SchneckenleitnerTakacs:2020}.
Using stability estimates for these extensions, we are able to provide a condition number estimate for the preconditioned Schur complement formulation of overall problem.

The remainder of this paper is organized as follows.
In Section~\ref{sec:2}, we introduce the model problem and discuss its
discretization.
The IETI-DP method is constructed in Section~\ref{sec:3}.
In Section~\ref{sec:4}, we give estimates characterizing the required discrete biharmonic extensions.
Using these results, we analyze an IETI-DP solver in Section~\ref{sec:5}.
In Section~\ref{sec:6}, we illustrate our theoretical findings with numerical results.
We close with some conclusions in Section~\ref{sec:7}.

\section{The model problem and its discretization}
\label{sec:2}

As model problem, we consider the \emph{first biharmonic problem}.
The computational domain $\Omega \subset \mathbb R^2$ is assumed to be an open and bounded domain with Lipschitz boundary $\partial \Omega$.
Given a sufficiently smooth function $f:\Omega\to \mathbb R$,
we want to find a sufficiently smooth function $u:\Omega \to \mathbb R$ such that
\begin{align}\label{eq:biharmonic1}
	\Delta^2 u  = f &\quad \mbox{in}\quad \Omega
	\qquad\mbox{and}\qquad
	u=\partial_n u=0 \quad \mbox{on}\quad \partial\Omega,
\end{align}
where $\partial_n$ denotes the directional derivative in the direction of the outer normal vector on $\partial\Omega$.
Since non-homogeneous boundary data can be handled using homogenization, we assume homogeneous boundary data for brevity.
The conforming variational formulation is formulated in the Sobolev space $H^2_0(\Omega)$.
Here and in what follows, $L^2(\Omega)$ and $H^r(\Omega)$ are the standard Lebesque and Sobolev spaces with scalar products $(\cdot,\cdot)_{L^2(\Omega)}$, $(\cdot,\cdot)_{H^r(\Omega)}$, norms $\|\cdot\|_{L^2(\Omega)}=(\cdot,\cdot)_{L^2(\Omega)}^{1/2}$, $\|\cdot\|_{H^r(\Omega)}=(\cdot,\cdot)_{H^r(\Omega)}^{1/2}$ and
seminorms $|\cdot|_{H^r(\Omega)}$. The same notation is used for subdomains, as well as for boundaries. The space $H^r_0(\Omega)$ is the closure of $C_0^\infty(\Omega)$ with respect to the norm $\|\cdot\|_{H^r(\Omega)}$. Following the Sobolev embedding theorem, functions in $H^2_0(\Omega)$ satisfy the boundary conditions~$u=\partial_n u=0$.
Analogously, $L^\infty(\Omega)$ is the set of all (essentially) bounded functions and $\|\cdot\|_{L^\infty(\Omega)}$ is the essential supremum of their absolute value. Moreover, we use the $L_0^\infty$-seminorm, given by $|v|_{L_0^\infty(\Omega)}:=\inf_{c \in\mathbb R} \|v-c\|_{L^\infty(\Omega)}$.

The \emph{variational formulation} is as follows. Given $f\in L^2(\Omega)$,
\begin{align}\label{eq:biharmonic2}
	\mbox{find }u\in H^2_0(\Omega):\qquad
	\int_\Omega \Delta u \Delta v \, \mathrm d \xx
	= \int_\Omega f\, v \, \mathrm d \xx
	\quad \mbox{for all} \quad v\in H^2_0(\Omega).
\end{align}
For each $V\subset \Omega$ with $\partial V\cap \partial\Omega=\emptyset$, the nullspace of $\int_V \Delta u \Delta u \, \mathrm d \xx$ contains all harmonic functions.
So, the local stiffness matrices for a domain decomposition solver (which realize such an integral) have a non-trivial nullspace, which consists of all functions in the function space used for discretization that exactly satisfy $\Delta u=0$. If B-splines are used for discretization, the dimension of this space is independent of the grid size, but grows with the spline degree $p$, cf. the discussion of the rank deficit in Ref.~\refcite{Kapl:2025}.
Instead, we use a different formulation. Using integration by parts, one immediately obtains (cf. also Ref.~\refcite{Grisvard:2011}) that~\eqref{eq:biharmonic2} is equivalent to
\begin{align}\label{eq:biharmonic3}
	\mbox{find }u\in H^2_0(\Omega):\qquad
	\int_\Omega \nabla^2 u : \nabla^2 v \, \mathrm d \xx
	= \int_\Omega f\, v \, \mathrm d \xx
	\quad \mbox{for all} \quad v\in H^2_0(\Omega),
\end{align}
where $\nabla^2$ denotes the Hessian and $\nabla^2 u : \nabla^2 v = \partial_{xx} u \,\partial_{xx} v+2\partial_{xy} u\,\partial_{xy}v + \partial_{yy} u\, \partial_{yy} v$ their Frobenius product. Here, the nullspace of $\int_V \nabla^2 u : \nabla^2 v \, \mathrm d \xx$ with $V\subset \Omega$ and $\partial V\cap \partial\Omega=\emptyset$, contains the affine-linear functions only.

We assume that the \emph{multi-patch computational domain} $\Omega$ is composed of $K$ non-overlapping
patches $\Omega_k$, i.e., we have
\begin{align}\label{eq:ass:geo}
	\closure{\Omega} = \bigcup_{k=1}^K \closure{\Omega_k} \quad \text{and}\quad
	\Omega_k \cap \Omega_\ell = \emptyset \quad\text{for all}\quad k \neq \ell.
\end{align}
Here and in what follows, $\closure{T}$ denotes the closure of any domain $T$.
Each patch $\Omega_k$ is parameterized by a bijective geometry function
\begin{align}
	G_k:\wh{\Omega}:=(0,1)^2 \rightarrow \Omega_k=G_k(\wh{\Omega}) \subset \mathbb{R}^2,
\end{align}
which can be continuously extended to the closure of the parameter domain
$\wh{\Omega}$. In IgA, the geometry function is typically represented as a linear combination of B-splines or NURBS, however this is actually not required for the theory we are developing.

We assume a \emph{geometrically conforming} setup, so for any $k\ne \ell$, the set
\begin{align}\label{eq:ass:geo2}
\mbox{$\Gamma_{k,\ell} := \partial{\Omega_k} \cap \partial{\Omega_\ell}$
is a common edge,
a common corner, or empty.}
\end{align}
As edges and corners of $\Omega_k$, we understand the images of the edges (including corners) and corners of the unit square $\wh \Omega$ under $G_k$.
Consequently, each boundary section
\begin{align}\label{eq:ass:geo3}
\mbox{ $\Gamma_{k,\partial\Omega}:=\partial \Omega_k \cap \partial \Omega$ is the union of zero of more edges and/or corners of $\Omega_k$.}
\end{align}

We assume a \emph{smooth parametrization}, i.e., that the first and second derivatives of $G_k$ and the inverse of the Jacobian of $G_k$ are uniformly bounded, i.e., there is a constant $C_G>0$ such that
\begin{equation}\label{eq:ass:nabla}
		\| \nabla G_k \|_{L^\infty(\wh{\Omega})} \le C_G\, H_k,
		\quad
		\| \nabla^2 G_k \|_{L^\infty(\wh{\Omega})} \le C_G\, H_k,
		\quad
		\| (\nabla G_k)^{-1} \|_{L^\infty(\wh{\Omega})} \le C_G\, H_k^{-1}
\end{equation}
holds for each patch $\Omega_k$ with diameter $H_k$. We assume that the same estimate also holds for the $L^\infty(\partial\wh\Omega)$-norm.
Standard trace estimates, cf. Ref.~\refcite{Lions:1972}, yield
\begin{equation}\label{eq:ass:trace}
	|v|_{H^{1/2}(\partial\Omega_k)}^2
	\le C_T
	|v|_{H^1(\Omega_k)}^2
 	\quad\mbox{for all}\quad v\in H^1(\Omega_k),
\end{equation}
with some constant $C_T>0$ that only depends on the shape of $\Omega_k$.

\begin{remark}
	The dependence of the terms in condition~\eqref{eq:ass:nabla} on the diameter $H_k$ is motivated as follows. For each patch $\Omega_k$, we define its shape via $\widetilde \Omega_k = H_k^{-1} \Omega_k$, which is parameterized by $\widetilde G_k(\wh{\mathbf x}):=H_k^{-1} G_k(\wh{\mathbf x})$. Assuming
	\begin{equation}\label{eq:ass:nabla:tilde}
			\| \nabla \widetilde G_k \|_{L^\infty(\wh{\Omega})} \le C_G,
			\quad
			\| \nabla^2 \widetilde G_k \|_{L^\infty(\wh{\Omega})} \le C_G,
			\quad
			\| (\nabla \widetilde G_k)^{-1} \|_{L^\infty(\wh{\Omega})} \le C_G,
	\end{equation}
	the condition~\eqref{eq:ass:nabla} follows from simple scaling arguments.
	Analogously, the constant $C_T$ from~\eqref{eq:ass:trace} only depends on $\widetilde \Omega_k$, not on $H_k$.
\end{remark}

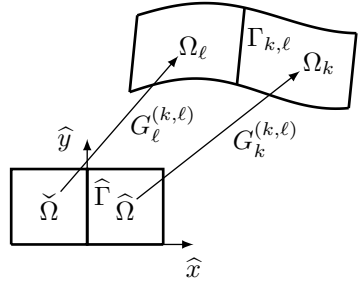
\begin{figure}
	\begin{center}

			\begin{tikzpicture}

			\draw[-,line width = 1pt](0.558824,3.01997)--
(0.617647,3.03973)--
(0.676471,3.0591)--
(0.735294,3.07788)--
(0.794118,3.09589)--
(0.852941,3.11293)--
(0.911765,3.12884)--
(0.970588,3.14347)--
(1.02941,3.15667)--
(1.08824,3.16829)--
(1.14706,3.17824)--
(1.20588,3.18641)--
(1.26471,3.19271)--
(1.32353,3.19709)--
(1.38235,3.1995)--
(1.44118,3.19991)--
(1.5,3.19833)--
(1.55882,3.19477)--
(1.61765,3.18926)--
(1.67647,3.18186)--
(1.73529,3.17264)--
(1.79412,3.1617)--
(1.85294,3.14914)--
(1.91176,3.13509)--
(1.97059,3.11969)--
(2.02941,3.1031)--
(2.08824,3.08548)--
(2.14706,3.067)--
(2.20588,3.04785)--
(2.26471,3.02822)--
(2.32353,3.00832)--
(2.38235,2.98833)--
(2.44118,2.96845)--
(2.5,2.94889)--
(2.55882,2.92984)--
(2.61765,2.9115)--
(2.67647,2.89403)--
(2.73529,2.87763)--
(2.79412,2.86245)--
(2.85294,2.84864)--
(2.91176,2.83634)--
(2.97059,2.82568)--
(3.02941,2.81677)--
(3.08824,2.80968)--
(3.14706,2.80449)--
(3.20588,2.80126)--
(3.26471,2.80002)--
(3.32353,2.80077)--
(3.38235,2.80351)--
(3.44118,2.80822);

			\draw[-,line width = 1pt] (0.658824,4.03795)--
(0.717647,4.05562)--
(0.776471,4.07283)--
(0.835294,4.08944)--
(0.894118,4.10531)--
(0.952941,4.12031)--
(1.01176,4.13432)--
(1.07059,4.14722)--
(1.12941,4.1589)--
(1.18824,4.16927)--
(1.24706,4.17824)--
(1.30588,4.18574)--
(1.36471,4.19171)--
(1.42353,4.19609)--
(1.48235,4.19885)--
(1.54118,4.19997)--
(1.6,4.19944)--
(1.65882,4.19727)--
(1.71765,4.19346)--
(1.77647,4.18805)--
(1.83529,4.1811)--
(1.89412,4.17264)--
(1.95294,4.16276)--
(2.01176,4.15154)--
(2.07059,4.13906)--
(2.12941,4.12544)--
(2.18824,4.11078)--
(2.24706,4.09521)--
(2.30588,4.07885)--
(2.36471,4.06183)--
(2.42353,4.04431)--
(2.48235,4.02642)--
(2.54118,4.00832)--
(2.6,3.99014)--
(2.65882,3.97205)--
(2.71765,3.95418)--
(2.77647,3.9367)--
(2.83529,3.91974)--
(2.89412,3.90344)--
(2.95294,3.88794)--
(3.01176,3.87336)--
(3.07059,3.85983)--
(3.12941,3.84746)--
(3.18824,3.83634)--
(3.24706,3.82658)--
(3.30588,3.81825)--
(3.36471,3.81143)--
(3.42353,3.80616)--
(3.48235,3.80249)--
(3.54118,3.80045);

			\draw[-,line width = 1pt] (0.558824,3.01997)--(0.658824,4.03795);
			\draw[-,line width = 1pt] (3.44118,2.80822)--(3.54118,3.80045);
			\draw[-,line width = 1pt] (1.97059,3.11969)--
(2.07059,4.13906);

			\draw[-,line width = 1pt] (-1,1) -- (0,1) -- (0,2) -- (-1,2) -- (-1,1)  node at (-.5,1.5) {$\widecheck\Omega$};
			\draw[-,line width = 1pt] (1,1) -- (0,1) -- (0,2) -- (1,2) -- (1,1) node at (.5,1.5) {$\wh\Omega$};
			\draw[->, line width = .5pt, -latex] (0,1) -- (1.4,1) node at (1.4,.7) {$\wh x$};
			\draw[->, line width = .5pt, -latex] (0,1) -- (0,2.4) node at (-.3,2.4) {$\wh y$};

			\draw[->, line width = .5pt, -latex] (-.35,1.7) -- (1.2,3.5) node at (1.4,3.6) {$\Omega_\ell$};
			\draw[->, line width = .5pt, -latex] (.65,1.6) -- (2.8,3.3) node at (3.05,3.4) {$\Omega_k$};
			\draw node at (.2,1.7) {$\wh\Gamma$};
			\draw node at (2.4,3.7) {$\Gamma_{k,\ell}$};

			\draw node at (1,2.6) {$G_\ell^{(k,\ell)}$};
			\draw node at (2.35,2.4) {$G_k^{(k,\ell)}$};

			\end{tikzpicture}
			\captionof{figure}{Parametrization of two neighboring patches}
			\label{fig:1}

	\end{center}
\end{figure}

Next, we discuss our assumptions concerning the \emph{smoothness of the parametrization} across the interfaces.
Consider any two patches $\Omega_k$ and $\Omega_\ell$ sharing an edge $\Gamma_{k,\ell}$. Both patches are parameterized with geometry functions, for $\Omega_k$ specifically $G_k:\wh\Omega\to\Omega_k$, where the pre-image of $\Gamma_{k,\ell}$ is one of the four sides of the parameter domain $\wh\Omega$. We rotate the function around the center of $\wh\Omega$ such that the pre-image is $\wh\Gamma:=\{0\}\times [0,1]$. The composition of the rotation and $G_k$ is a bijective function $G_k^{(k,\ell)}: \wh\Omega \to \Omega_k$ with $G_k^{(k,\ell)}(\wh\Gamma) = \Gamma_{k,\ell}$ that also satisfies~\eqref{eq:ass:nabla}. Analogously, we construct using translation and rotation a bijective function $G_\ell^{(k,\ell)}: \widecheck\Omega\to \Omega_\ell$ with $\widecheck\Omega:=(-1,0)\times(0,1) $ and $G_\ell^{(k,\ell)}(\wh\Gamma) = \Gamma_{k,\ell}$ that also satisfies the same bounds as in~\eqref{eq:ass:nabla}. See Figure~\ref{fig:1} for a visualization.
We assume that
\begin{equation}\label{eq:ass:mathching}
	G_k^{(k,\ell)} = G_\ell^{(k,\ell)}
	\quad\mbox{and}\quad
	\partial_{\wh x} G_k^{(k,\ell)} = \partial_{\wh x} G_\ell^{(k,\ell)}
	\quad\mbox{on}\quad
	\wh \Gamma,
\end{equation}
i.e., that the parametrization of the geometry is continuously differentiable across the domains. Using this definition, the function $G^{(k,\ell)}:\closure({\widecheck\Omega\cup \wh\Omega}) \to \closure({\Omega_\ell\cup\Omega_k})$, given by
\begin{equation}\label{eq:gkl}
	G^{(k,\ell)}(\wh \xx)
	:=
	\begin{cases}
		G^{(k,\ell)}_k(\wh \xx) & \mbox{ if } \wh \xx \in \closure{\wh \Omega}, \\
		G^{(k,\ell)}_\ell (\wh \xx) & \mbox{ otherwise,}
	\end{cases}
\end{equation}
is  $C^1$ smooth and bijective.
\begin{remark}
	This assumption excludes the possibility of extraordinary vertices, i.e., vertices in the interior of $\Omega$ with only $3$ or with $5$ or more adjacent patches.
	If instead of the second condition in~\eqref{eq:ass:mathching}, the condition $\partial_{\wh x} G_k^{(k,\ell)} = \alpha_{k,\ell} \partial_{\wh x} G_\ell^{(k,\ell)}$ holds with some constant $\alpha_{k,\ell}>0$, an analogous analysis can be performed. This case is also considered in the numerical experiments in Section~\ref{sec:6:2}.
	More general configurations can be represented using analysis-suitable $G^1$ discretizations (cf. Refs.~\refcite{CollinSangalliTakacs:2016,KaplSangalliTakacs:2018}). An extension of our analysis to those settings is out of the scope of this paper and left as future work.
\end{remark}

Next, we define the \emph{isogeometric function spaces}.
We consider the one dimensional case first.
Let $p\in\{2,3,\dots\}$ be a spline degree. For any $p$-open knot vector $\Xi=(\xi_0,\dots,\xi_{N+p})$ with
\[
		0=\xi_0 = \dots = \xi_{p} < \xi_{p+1} \le \xi_{p+2} \le \dots \le
		\xi_{N-1}<\xi_{N}=\dots=\xi_{N+p} =1,
\]
we denote by $\Phi_{p,\Xi}=(\phi_{p,\Xi}^{(1)},\dots,\phi_{p,\Xi}^{(N)})$ the set of B-spline basis functions $(0,1)\to \mathbb R$ of degree $p$ over $\Xi$, as obtained by the Cox-de~Boor formula, cf. (2.1) and (2.2) in Ref.~\refcite{CottrellHughesBazilevs:2009}.
In order to guarantee that these spline spaces are in $H^2(0,1)$, or, equivalently, in $C^1[0,1]$, the multiplicity of each (inner) knot must not be larger than $p-1$.

We are interested in basis functions that interpolate the function values and the derivatives on the boundary. In order to achieve this, we assume
$N\ge 4$ (which is always the case for $p\ge 3$ and which is a consequence of assumption~\eqref{eq:hlequarter} below for $p=2$)
and introduce a \emph{basis transformation}; specifically, define $\Psi_{p,\Xi}:=(\psi_{p,\Xi}^{(1)},\dots,\psi_{p,\Xi}^{(N)})$ via
\begin{equation}\label{eq:psidef}
		\psi_{p,\Xi}^{(i)}
		:=
		\begin{cases}
			\phi_{p,\Xi}^{(1)} + \phi_{p,\Xi}^{(2)} & \mbox{ if } i=1 \\
			p^{-1} \xi_{p+1} \phi_{p,\Xi}^{(2)} & \mbox{ if } i=2 \\
			\phi_{p,\Xi}^{(i)} & \mbox{ if } 2< i < N-1 \\
			p^{-1} (1-\xi_{N-1}) \phi_{p,\Xi}^{(N-1)} & \mbox{ if } i=N-1 \\
			\phi_{p,\Xi}^{(N)} + \phi_{p,\Xi}^{(N-1)} & \mbox{ if } i=N .
		\end{cases}
\end{equation}
Using this choice, we have for all $i=1,\dots,N$ the identities
\begin{equation}\label{eq:psi:bdy}
		\psi_{p,\Xi}^{(i)}(0) = \delta_{i,1},
		\quad
		\partial_n \psi_{p,\Xi}^{(i)}(0) = \delta_{i,2},
		\quad
		\partial_n \psi_{p,\Xi}^{(i)}(1) = \delta_{i,N-1},
		\quad
		\psi_{p,\Xi}^{(i)}(1) = \delta_{i,N},
\end{equation}
where $\delta_{i,j}$ is the Kronecker-delta. The basis is visualized in Figure~\ref{fig:basis}: The basis functions with function value $1$ on the boundary are depicted in blue color, those with normal derivative $1$ are depicted in red color. The basis functions that coincide with the standard basis are depicted in black.
\begin{figure}
	\begin{center}
		\includegraphics[width=4cm]{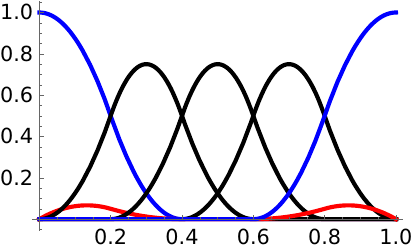}
			\captionof{figure}{Example of transformed basis}
			\label{fig:basis}
	\end{center}
\end{figure}

For the parameter domain $\wh \Omega$, we introduce \emph{tensor-product splines}.
For each patch $\Omega_k$, we choose two $p$-open knot vectors $\Xi_{k,\wh x}$ and $\Xi_{k,\wh y}$ with $N_{k,\wh x}+p+1$ and $N_{k,\wh y}+p+1$ knots, respectively.
For simplicity, we assume to have the same degree $p$ on all patches.
For $\mathbf\Xi_k:=(\Xi_{k,\wh x},\Xi_{k,\wh y})$, we denote by
\begin{equation}\label{eq:Psikdef}
	\wh\Psi_{p,\mathbf\Xi_k} :
	= \{ \wh\psi \in C^1(\wh\Omega) : \wh\psi(\wh x,\wh y)=\psi_{\wh x}(\wh x) \psi_{\wh y}(\wh y),\; \psi_{\wh x}\in \Psi_{p,\Xi_{k,\wh x}},\;\psi_{\wh y}\in \Psi_{p,\Xi_{k,\wh y}}\}
\end{equation}
the set of tensor-product B-spline basis functions. As function spaces, we choose
\begin{equation}\label{eq:Vkdef}
	\wh V_k := \{ v \in \mathrm{span}\,\Psi_{p,\Xi_k}
	\,:\, v|_{\wh \Gamma_{k,\partial\Omega}} =  \partial_n v|_{\wh \Gamma_{k,\partial\Omega}} = 0 \}
	\subset H^2(\wh\Omega),
\end{equation}
where $\wh \Gamma_{k,\partial\Omega}:=G_k^{-1}(\Gamma_{k,\partial\Omega})$ is the pre-image of $\Gamma_{k,\partial\Omega}$ as defined in~\eqref{eq:ass:geo3}. The function space for the physical patch is defined via the \emph{pull-back principle}:
\begin{equation}\label{eq:pullback}
	V_k := \{ v : v\circ G_k \in \wh V_k \}  \subset H^2_{0,\partial\Omega}(\Omega_k) :=\{ v\in H^2(\Omega_k):   v|_{\Gamma_{k,\partial\Omega}} = \partial_n v|_{\Gamma_{k,\partial\Omega}} = 0\}.
\end{equation}
Analogously, let $\Psi_k := \{ \psi : \psi \circ G_k \in \wh\Psi_{p,\mathbf\Xi_k} \cap H^2_{0,\partial\Omega}(\Omega_k)\}$ be the basis for $V_k$.

In order to be able to formulate the desired continuity conditions, we assume that for any two patches $\Omega_k$ and $\Omega_\ell$ sharing an edge $\Gamma_{k,\ell}$
\begin{equation}\label{eq:ass:knotvec matching}
	\mbox{the knot vectors associated to the direction of $\Gamma_{k,\ell}$ agree.}
\end{equation}
Discretizations satisfying this condition are also called \emph{fully matching}. This condition guarantees that for each basis function in $\Psi_k$, there is a basis function in $\Psi_\ell$ such that their traces on $\Gamma_{k,\ell}$ and the traces of the normal derivatives on $\Gamma_{k,\ell}$ agree. If the traces do not vanish, the basis function in $\Psi_\ell$ is unique.

Let $\wh h_k$ and $\wh h_{k,\mathrm{min}}$
be the largest and the smallest (non-empty) knot span of $\Xi_{k,1}$ and $\Xi_{k,2}$.
The grids are quasi-uniform; $C_Q>0$ is the largest constant such that
\begin{equation}\label{eq:ass:quasi uniform}
	\wh h_{k,\mathrm{min}}  \ge C_Q \wh h_k
 	\quad\mbox{for all}\quad k\in\{1,\ldots,K\}.
\end{equation}
The grid size on the physical domain is defined by scaling, i.e., via
$h_k:=H_k \wh{h}_k$
and
$h_{k,\mathrm{min}}:=H_k \wh{h}_{k,\mathrm{min}}$.
In order to allow for our theoretical estimates, assume further that
\begin{equation}\label{eq:hlequarter}
			 \wh h_k \le \frac14 \quad\mbox{for all}\quad k\in\{1,\ldots,K\}.
\end{equation}

A conforming Galerkin discretization of the variational
problem~\eqref{eq:biharmonic3} using the overall function space
\begin{equation}\label{eq:vdef}
V_{\mathrm{conf}} := \lbrace u \in C^1(\Omega): u|_{\Omega_k} \in V_k \text{ for } k = 1, \dots, K \rbrace \subset H^2_0(\Omega).
\end{equation}
reads as
\begin{align}\label{eq:biharmonic3h}
	\mbox{find }u \in V_{\mathrm{conf}}:\qquad
	\underbrace{\int_\Omega \nabla^2 u : \nabla^2 v \, \mathrm d \xx}_{\displaystyle a(u,v):=}
	= \underbrace{\int_\Omega f\, v \, \mathrm d \xx}_{\displaystyle \langle f,v\rangle:=}
	\quad \mbox{for all} \quad v\in V_{\mathrm{conf}},
\end{align}
which is equivalent to the minimization problem
\begin{align}\label{eq:biharmonic3min}
	\frac12 a(u,u) -\langle f,u\rangle \to \min, \quad u\in V_{\mathrm{conf}}.
\end{align}

\section{The setup of the IETI-DP solver}
\label{sec:3}

In this section, we derive a IETI-DP solver for the variational problem~\eqref{eq:biharmonic3h}. First, we assemble $a=a_1+\dots+a_K$ and $f=f_1+\dots+f_K$ locally via
\begin{align*}
	a_k(u,v) :=
	\int_{\Omega_k} \nabla^2 u : \nabla^2 v \, \mathrm d \xx
	\quad\mbox{and}\quad
	\langle f_k,v \rangle :=
	 \int_{\Omega_k} f\, v \, \mathrm d \xx.
\end{align*}
The variational problem~\eqref{eq:biharmonic3h} and the equivalent minimization problem \eqref{eq:biharmonic3min} are equivalent to the problem
\begin{equation}\label{eq:3:minimize0}
\sum_{k=1}^K  \big( \tfrac12 a_k(u^{(k)},u^{(k)}) - \langle f_k, u^{(k)} \rangle \big) \to \min, \; u^{(k)} \in V_k
\end{equation}
subject to the smoothness conditions
\begin{equation}\label{eq:3:minimize constr}
		u^{(k)} = u^{(\ell)}, \qquad \partial_{n_k} u^{(k)} = \partial_{n_k} u^{(\ell)} \qquad \mbox{ on } \qquad \Gamma_{k,\ell} =\partial\Omega_k\cap \partial\Omega_\ell
\end{equation}
for any two patches $\Omega_k$ and $\Omega_\ell$ sharing an edge and with $k< \ell$.

Using the basis $\Psi_{k}=: (\psi^{(k)}_1,\ldots,\psi^{(k)}_{N_k})$ of $V_k$, every function $u^{(k)}\in V_k$ can be uniquely represented as linear combination of the basis functions, i.e., we have $u^{(k)} = \sum_{i=1}^{N_k} u_i^{(k)} \psi^{(k)}_i$ with coefficients $u_i^{(k)}$ that can be collected into a vector $\ul u^{(k)}:=(u_1^{(k)},\dots,u_{N_k}^{(k)})^\top$. Also using the same basis, we define local stiffness matrices and load vectors:
\[
	A^{(k)} = [a_k(\psi_j^{(k)},\psi_i^{(k)})]_{i,j=1}^{N_k} \quad\mbox{and}\quad
	\ul{f}^{(k)} = [ \langle f_k,\psi_i^{(k)}\rangle]_{i=1}^{N_k}.
\]
\eqref{eq:gkl} states that the union of any two patches $\Omega_k$ and $\Omega_\ell$ sharing an edge can be parameterized with a function $G^{(k,\ell)}: [-1,1]\times[0,1] \to \closure{(\Omega_k\cup\Omega_\ell)}$. The smoothness conditions~\eqref{eq:3:minimize constr} are equivalent to the smoothness conditions
\begin{align}\label{eq:3:smoothness1}
		\wh u^{(k)} &= \wh u^{(\ell)}, && \mbox{ on } \quad \wh \Gamma,\\
		\label{eq:3:smoothness2}
		\partial_{\wh x} \wh u^{(k)} &= \partial_{\wh x} \wh u^{(\ell)} && \mbox{ on } \quad \wh \Gamma,
\end{align}
where $\wh u^{(k)} : \wh\Omega \to \mathbb R$ and $\wh u^{(\ell)}: \widecheck\Omega \to \mathbb R$ are the pre-images of $u^{(k)}$ and $u^{(\ell)}$, respectively, under $G^{(k,\ell)}$. Note that the pre-image of the normal vector $n_k$ is not necessarily equal to $\wh x$, however the combination of the conditions~\eqref{eq:3:smoothness1} and \eqref{eq:3:smoothness2} guarantees that the gradients are continuous across the interfaces.  This is also guaranteed by~\eqref{eq:3:minimize constr}.

The enforcement of the condition~\eqref{eq:3:smoothness1} is standard. According to~\eqref{eq:psi:bdy}, only one basis function of the transformed B-spline basis is nonzero on the boundary. So, for the tensor-product basis, only the outermost ``layer'' of basis functions is active on the interface and the traces of these basis functions are linear independent.
So, for all basis functions $\psi_i^{(k)}$ and $\psi_j^{(\ell)}$ with pullbacks $\wh\psi_i^{(k)}$ and $\wh\psi_j^{(\ell)}$ under $G^{(k,\ell)}$ such that
\[
\wh\psi_i^{(k)}|_{\wh\Gamma} = \wh\psi_j^{(\ell)}|_{\wh\Gamma}\not=0
,
\]
we introduce a constraint of the form
\begin{equation}\label{eq:def:b}
u^{(k)}_{i} - u^{(\ell)}_{j} = 0.
\end{equation}
These constraints are visualized for the case of two patches in Figure~\ref{fig:schema} by arrows connecting the blue square dots. The extension to the case of several patches is discussed below.
\begin{figure}
	\centering
	\scalebox{.45}{
		\begin{tikzpicture}
		\definecolor{lightgray}{gray}{0.95}
		\filldraw [color=lightgray] (0,0) rectangle (5,7);
		\draw (0,0) -- (5,0) -- (5,7) -- (0,7) -- (0,0);

		\foreach \i in {0,...,7} \foreach \j in {4,5}
		\node[color=gray,mark size=3.5pt] at (\j,\i) {\pgfuseplotmark{*}};

		\foreach \i in {0,1,6,7}\foreach \j in {0,...,3}
		\node[color=gray,mark size=3.5pt] at (\j,\i) {\pgfuseplotmark{*}};

		\foreach \i in {2,...,5}
		\node[color=blue,mark size=3.5pt] at (0,\i) {\pgfuseplotmark{square*}};

		\foreach \i in {2,...,5}
		\node[color=red,mark size=3.5pt] at (1,\i) {\pgfuseplotmark{triangle*}};

		\foreach \i in {2,...,5} \foreach \j in {2,...,3}
		\node[color=black,mark size=3.5pt] at (\j,\i) {\pgfuseplotmark{*}};

		\filldraw [color=lightgray] (-1,0) rectangle (-6,7);
		\draw (-1,0) -- (-6,0) -- (-6,7) -- (-1,7) -- (-1,0);

		\foreach \i in {0,...,7} \foreach \j in {-5,-6}
		\node[color=gray,mark size=3.5pt] at (\j,\i) {\pgfuseplotmark{*}};

		\foreach \i in {0,1,6,7}\foreach \j in {-1,...,-4}
		\node[color=gray,mark size=3.5pt] at (\j,\i) {\pgfuseplotmark{*}};

		\foreach \i in {2,...,5}
		\node[color=blue,mark size=3.5pt] at (-1,\i) {\pgfuseplotmark{square*}};

		\foreach \i in {2,...,5}
		\node[color=red,mark size=3.5pt] at (-2,\i) {\pgfuseplotmark{triangle*}};

		\foreach \i in {2,...,5} \foreach \j in {-3,...,-4}
		\node[color=black,mark size=3.5pt] at (\j,\i) {\pgfuseplotmark{*}};

		\foreach \j in {1.9,2.9,3.9,4.9}
		\draw [{Stealth[length=3mm, width=2.5mm]}-{Stealth[length=3mm, width=2.5mm]}, thick] (0,\j) to [out=-150,in=-30]  (-1,\j);

		\foreach \j in {2.1,3.1,4.1,5.1}
		\draw [{Stealth[length=3mm, width=2.5mm]}-{Stealth[length=3mm, width=2.5mm]}, thick] (1,\j) to [out=150,in=30]  (-2,\j);

		\end{tikzpicture}
	}
	\caption{\label{fig:schema} Constraints between two patches.}
\end{figure}
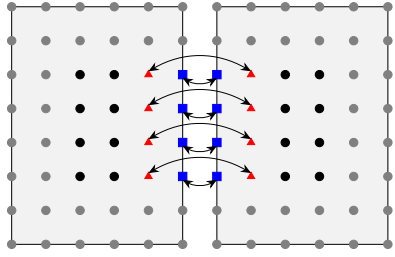
Additionally, we have to enforce the smoothness condition~\eqref{eq:3:smoothness2} or, equivalently,
$\partial_{\wh n_k} \wh u_k+ \partial_{\wh n_\ell} \wh u_\ell=0$, where
$\wh n_k=-\wh n_\ell=(-1,0)^\top$ are the corresponding normal vectors.
According to~\eqref{eq:psi:bdy}, there is again one ``layer'' of basis functions whose derivatives do not vanish on the interfaces, i.e., such that
\[
	\partial_{\wh n_k} \wh \psi_i^{(k)}|_{\wh \Gamma} = - \partial_{\wh n_\ell} \wh \psi_j^{(\ell)}|_{\wh \Gamma}\not=0.
\]
For each of these basis functions, we introduce smoothness conditions of the form
\begin{equation}\label{eq:def:b2}
		u^{(k)}_{i}
		+
		u^{(\ell)}_{j}
		=
		0.
\end{equation}
These constraints are visualized in Figure~\ref{fig:schema} by arrows connecting the red triangular dots.
The black dots correspond to ``inner'' basis functions whose function value and whose derivative vanish on $\partial \Omega_k$. The gray dots correspond to the basis functions that are not in $H^2_{0,\partial\Omega}(\Omega_k)$ and $H^2_{0,\partial\Omega}(\Omega_\ell)$, so they are  not in the bases $\Psi_k$ and $\Psi_\ell$.

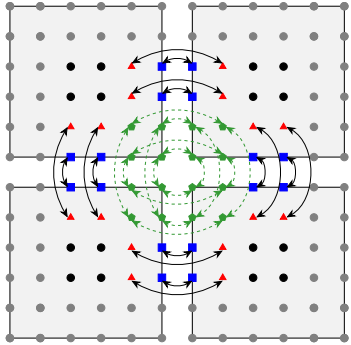
\begin{figure}[b]
	\centering
	\scalebox{.4}{
		\begin{tikzpicture}
		\definecolor{lightgray}{gray}{0.95}
		\definecolor{green}{rgb}{0.2, 0.6, 0.2}

		\filldraw [color=lightgray] (0,0) rectangle (5,5);
		\draw (0,0) -- (5,0) -- (5,5) -- (0,5) -- (0,0);
		\filldraw [color=lightgray] (-1,0) rectangle (-6,5);
		\draw (-1,0) -- (-6,0) -- (-6,5) -- (-1,5) -- (-1,0);
		\filldraw [color=lightgray] (0,-1) rectangle (5,-6);
		\draw (0,-1) -- (5,-1) -- (5,-6) -- (0,-6) -- (0,-1);
		\filldraw [color=lightgray] (-1,-1) rectangle (-6,-6);
		\draw (-1,-1) -- (-6,-1) -- (-6,-6) -- (-1,-6) -- (-1,-1);

		\foreach \i in {0,...,5} \foreach \j in {4,5}
		{
		\node[color=gray,mark size=3.5pt] at (\j,\i) {\pgfuseplotmark{*}};
		\node[color=gray,mark size=3.5pt] at (\i,\j) {\pgfuseplotmark{*}};
		\node[color=gray,mark size=3.5pt] at (-1-\j,\i) {\pgfuseplotmark{*}};
		\node[color=gray,mark size=3.5pt] at (-1-\i,\j) {\pgfuseplotmark{*}};
		\node[color=gray,mark size=3.5pt] at (\j,-1-\i) {\pgfuseplotmark{*}};
		\node[color=gray,mark size=3.5pt] at (\i,-1-\j) {\pgfuseplotmark{*}};
		\node[color=gray,mark size=3.5pt] at (-1-\j,-1-\i) {\pgfuseplotmark{*}};
		\node[color=gray,mark size=3.5pt] at (-1-\i,-1-\j) {\pgfuseplotmark{*}};
		}
		\foreach \i in {2,...,3} \foreach \j in {2,...,3}
		{
		\node[color=black,mark size=3.5pt] at (\j,\i) {\pgfuseplotmark{*}};
		\node[color=black,mark size=3.5pt] at (-1-\j,\i) {\pgfuseplotmark{*}};
		\node[color=black,mark size=3.5pt] at (\j,-1-\i) {\pgfuseplotmark{*}};
		\node[color=black,mark size=3.5pt] at (-1-\j,-1-\i) {\pgfuseplotmark{*}};
		}

		\foreach \i in {2,...,3}
		{
		\node[color=blue,mark size=3.5pt] at (0,\i) {\pgfuseplotmark{square*}};
		\node[color=blue,mark size=3.5pt] at (\i,0) {\pgfuseplotmark{square*}};
		\node[color=blue,mark size=3.5pt] at (-1-0,\i) {\pgfuseplotmark{square*}};
		\node[color=blue,mark size=3.5pt] at (-1-\i,0) {\pgfuseplotmark{square*}};
		\node[color=blue,mark size=3.5pt] at (0,-1-\i) {\pgfuseplotmark{square*}};
		\node[color=blue,mark size=3.5pt] at (\i,-1-0) {\pgfuseplotmark{square*}};
		\node[color=blue,mark size=3.5pt] at (-1-0,-1-\i) {\pgfuseplotmark{square*}};
		\node[color=blue,mark size=3.5pt] at (-1-\i,-1-0) {\pgfuseplotmark{square*}};
		}

		\foreach \i in {2,...,3}
		{
		\node[color=red,mark size=3.5pt] at (1,\i) {\pgfuseplotmark{triangle*}};
		\node[color=red,mark size=3.5pt] at (\i,1) {\pgfuseplotmark{triangle*}};
		\node[color=red,mark size=3.5pt] at (-1-1,\i) {\pgfuseplotmark{triangle*}};
		\node[color=red,mark size=3.5pt] at (-1-\i,1) {\pgfuseplotmark{triangle*}};
		\node[color=red,mark size=3.5pt] at (1,-1-\i) {\pgfuseplotmark{triangle*}};
		\node[color=red,mark size=3.5pt] at (\i,-1-1) {\pgfuseplotmark{triangle*}};
		\node[color=red,mark size=3.5pt] at (-1-1,-1-\i) {\pgfuseplotmark{triangle*}};
		\node[color=red,mark size=3.5pt] at (-1-\i,-1-1) {\pgfuseplotmark{triangle*}};
		}

		\foreach \i in {-2,...,1}
		\foreach \j in {-2,...,1}
		\node[color=green,mark size=3.5pt] at (\i,\j) {\pgfuseplotmark{pentagon*}};

		\foreach \j in {2.1,3.1}
		{
		\draw [{Stealth[length=3mm, width=2.5mm]}-{Stealth[length=3mm, width=2.5mm]}, thick] (0,\j) to [out=150,in=30]  (-1,\j);
		\draw [{Stealth[length=3mm, width=2.5mm]}-{Stealth[length=3mm, width=2.5mm]}, thick] (\j,0) to [out=-30,in=30]  (\j,-1);
		\draw [{Stealth[length=3mm, width=2.5mm]}-{Stealth[length=3mm, width=2.5mm]}, thick] (1,\j) to [out=150,in=30]  (-2,\j);
		\draw [{Stealth[length=3mm, width=2.5mm]}-{Stealth[length=3mm, width=2.5mm]}, thick] (\j,1) to [out=-30,in=30]  (\j,-2);
		}

		\foreach \j in {-4.1,-3.1}
		{
		\draw [{Stealth[length=3mm, width=2.5mm]}-{Stealth[length=3mm, width=2.5mm]}, thick] (0,\j) to [out=-150,in=-30]  (-1,\j);
		\draw [{Stealth[length=3mm, width=2.5mm]}-{Stealth[length=3mm, width=2.5mm]}, thick] (\j,0) to [out=-120,in=-240]  (\j,-1);
		\draw [{Stealth[length=3mm, width=2.5mm]}-{Stealth[length=3mm, width=2.5mm]}, thick] (1,\j) to [out=-150,in=-30]  (-2,\j);
		\draw [{Stealth[length=3mm, width=2.5mm]}-{Stealth[length=3mm, width=2.5mm]}, thick] (\j,1) to [out=-120,in=-240]  (\j,-2);
		}

		\foreach \j in {0.1,1.1}
		{
		\draw [{Stealth[length=3mm, width=2.5mm]}-{Stealth[length=3mm, width=2.5mm]}, thick, color=green, dashed] (0,\j) to [out=150,in=30]  (-1,\j);
		\draw [{Stealth[length=3mm, width=2.5mm]}-{Stealth[length=3mm, width=2.5mm]}, thick, color=green, dashed] (\j,0) to [out=-30,in=30]  (\j,-1);
		\draw [{Stealth[length=3mm, width=2.5mm]}-{Stealth[length=3mm, width=2.5mm]}, thick, color=green, dashed] (1,\j) to [out=150,in=30]  (-2,\j);
		\draw [{Stealth[length=3mm, width=2.5mm]}-{Stealth[length=3mm, width=2.5mm]}, thick, color=green, dashed] (\j,1) to [out=-30,in=30]  (\j,-2);
		}

		\foreach \j in {-2.1,-1.1}
		{
		\draw [{Stealth[length=3mm, width=2.5mm]}-{Stealth[length=3mm, width=2.5mm]}, thick, color=green, dashed] (0,\j) to [out=-150,in=-30]  (-1,\j);
		\draw [{Stealth[length=3mm, width=2.5mm]}-{Stealth[length=3mm, width=2.5mm]}, thick, color=green, dashed] (\j,0) to [out=-120,in=-240]  (\j,-1);
		\draw [{Stealth[length=3mm, width=2.5mm]}-{Stealth[length=3mm, width=2.5mm]}, thick, color=green, dashed] (1,\j) to [out=-150,in=-30]  (-2,\j);
		\draw [{Stealth[length=3mm, width=2.5mm]}-{Stealth[length=3mm, width=2.5mm]}, thick, color=green, dashed] (\j,1) to [out=-120,in=-240]  (\j,-2);
		}

		\end{tikzpicture}
	}
	\caption{\label{fig:schema2} Constraints between four patches.}
\end{figure}

Having several patches, the same construction is applied on an edge-by-edge basis. Special care has to be taken for the degrees of freedom at the cross points, i.e., to degrees of freedom that are associated to more than two patches. In our application, these degrees of freedom are associated to the corners of the patches. Specifically, there are four of those degrees of freedom for each corner. This case is visualized in Figure~\ref{fig:schema2}, where green pentagons are used to indicate the corner degrees of freedom, dashed green lines are used to indicate the constraints that would be obtained by applying the construction of~\eqref{eq:def:b} and~\eqref{eq:def:b2} in a straight-forward way to those degrees of freedom.
We do not represent them using constraints, instead we choose the corner degrees of freedom as primal degrees of freedom.
This means that we assign global indices to each of the corner degrees of freedom.
Since these degrees of freedom not only include the corner values, but also the values of the derivatives on the corners, this choice of primal degrees of freedom is also known as \emph{fat vertices}, cf. Ref.~\refcite{BdVPavarinoScacciWidlundZampini:2014}.
We denote the number of these global degrees of freedom by $N_\Pi$.
For each patch $\Omega_k$, we have a matrix $C^{(k)} \in \{0,1\}^{N_{k,\Pi}\times N_k}$ that selects the primal degrees of freedom (depicted with green pentagon dots), where $N_{k,\Pi}$ is the number of primal degrees of freedom belonging to that patch. Moreover, we have a matrix $ R^{(k)}\in \{-1,0,1\}^{N_{k,\Pi}\times N_\Pi}$ that associates the global indices of the primal degrees of freedom with their patch local indices. The negative coefficients in $R^{(k)}$ account for the different signs of the normal vectors. The $C^1$-smoothness condition for the primal degrees of freedom is guaranteed by the constraint
\begin{align}
	\label{eq:3:contC}
	\mbox{there is a }\; \ul w_\Pi \in\mathbb R^{N_\Pi} \;\mbox{ such that }\;
	 C^{(k)} \ul u^{(k)} = R^{(k)} \ul w_\Pi
	\;\mbox{ for } k=1,\ldots,K.
\end{align}
The remaining constraints (depicted with black solid arrows) both of the form \eqref{eq:def:b} and the form~\eqref{eq:def:b2} are collected row-wise into matrices $ B^{(k)}\in \{-1,0,1\}^{N_\lambda\times N_k}$ such that each constraint is a row of the linear system
\begin{align}
	\label{eq:3:contB}
	\sum_{k=1}^K  B^{(k)} \ul u^{(k)} = 0
\end{align}
and $N_\lambda$ is the number of the constraints.
The problem~\eqref{eq:3:minimize0} -- and thus~\eqref{eq:biharmonic3h} -- is equivalent to \begin{equation}\label{eq:3:minimize2}
\sum_{k=1}^K \big( \tfrac12 \ul u^{(k)\top} A^{(k)} \ul u^{(k)} - \ul u^{(k)\top}\ul f^{(k)} \big) \to \min \quad \mbox{subject to~\eqref{eq:3:contC} and~\eqref{eq:3:contB}.}
\end{equation}
The optimality system for this minimization problem is
\[
	\begin{pmatrix}
		A^{(1)} &  C^{(1)\top} &     &      &&     0 &  B^{(1)\top} \\
		 C^{(1)} & 0&     &&      &     -R^{(1)} & 0 \\
		& & \ddots &  && \vdots& \vdots \\
		&&& A^{(K)} &  C^{(K)\top} &     0 &  B^{(K)\top} \\
		&&&  C^{(K)}      &0&     -R^{(K)} & 0 \\
		0 & -R^{(1)\top}  & \hdots & 0 & -R^{(K)\top} \\
		 B^{(1)} & 0  & \hdots &  B^{(K)} & 0 \\
	\end{pmatrix}
	\begin{pmatrix}
		\ul u^{(1)} \\ \ul \mu^{(1)} \\
		\vdots \\
		\ul u^{(K)} \\ \ul \mu^{(K)} \\
		\ul w_{\Pi} \\ \ul \lambda
	\end{pmatrix}
	=
	\begin{pmatrix}
		\ul f^{(1)} \\ 0 \\
		\vdots \\
		\ul f^{(K)} \\ 0 \\
		0 \\ 0
	\end{pmatrix}
\]
In order to get a more explicit representation, we group the degrees of freedom into dual degrees of freedom (index ``$\Delta$'') and primal degrees of freedom (index ``$\Pi$''). The primal degrees of freedom are those associated to the corners (green pentagon dofs in Figure~\ref{fig:schema2}), the dual degrees of freedom are the remaining (free) degrees of freedom (black circular dots, blue triangular dots and red square dots). Assuming that the dofs are ordered according to their assignment to these groups, we have
\begin{align*}
	A^{(k)}
	&=
	\begin{pmatrix}
	A_{\Delta\Delta}^{(k)} & A_{\Delta\Pi}^{(k)}\\
	A_{\Pi\Delta}^{(k)} & A_{\Pi\Pi}^{(k)}\\
	\end{pmatrix},
	\;
	\ul u^{(k)}
	=
	\begin{pmatrix}
	\ul u_{\Delta}^{(k)} \\ \ul u_{\Pi}^{(k)}\\
	\end{pmatrix},
	\;
	\ul f^{(k)}
	=
	\begin{pmatrix}
	\ul f_{\Delta}^{(k)} \\ \ul f_{\Pi}^{(k)}\\
	\end{pmatrix},
	\;
	\begin{array}{c}
		 B^{(k)}
		=
		\begin{pmatrix}
		B_{\Delta}^{(k)} & 0\\
		\end{pmatrix},
		\\
		 C^{(k)}
		=
		\begin{pmatrix}
		0 & C_{\Pi}^{(k)}
	\end{pmatrix}.
	\end{array}
\end{align*}
Assuming an appropriate ordering of the dofs, we further have $C_\Pi^{(k)}=I$.
By substituting this into the optimality system and by eliminating the $\ul\mu^{(k)}$, we obtain
\[
	\begin{pmatrix}
		A_{\Delta\Delta}^{(1)}&     &      &     A_{\Delta\Pi}^{(1)}  R^{(1)}& B_\Delta^{(1)\top} \\
		 & \ddots &  & \vdots& \vdots \\
		&&  A_{\Delta\Delta}^{(K)}&  A_{\Delta\Pi}^{(K)}  R^{(K)} & B_\Delta^{(K)\top}\\
		R^{(1)\top} A_{\Pi \Delta}^{(1)}  & \hdots &  R^{(K)\top} A_{\Pi \Delta}^{(K)} &  A_{\Pi\Pi}\\
		B_\Delta^{(1)}  & \hdots &  B_\Delta^{(K)} & \\
	\end{pmatrix}
	\begin{pmatrix}
		\ul u_\Delta^{(1)} \\
		\vdots \\
		\ul u_\Delta^{(K)} \\
		\ul w_{\Pi} \\ \ul \lambda
	\end{pmatrix}
	=
	\begin{pmatrix}
		\ul f_\Delta^{(1)}  \\
		\vdots \\
		\ul f_\Delta^{(K)}  \\
		0 \\ 0
	\end{pmatrix},
\]
where $A_{\Pi\Pi}:=\sum_{k=1}^K R^{(k)\top} A_{\Pi\Pi}^{(k)}R^{(k)}$ is obtained by assembling the contributions from all patches.
Using a basis transformation guaranteeing the orthogonality of primal and dual basis functions, we obtain
\begin{equation}\label{eq:3:kkt2}
	\begin{pmatrix}
		A_{\Delta\Delta}^{(1)}&     &      &     & B_\Delta^{(1)\top} \\
		 & \ddots &  & & \vdots \\
		&&  A_{\Delta\Delta}^{(K)}&   & B_\Delta^{(K)\top}\\
		&  &  &  S_{\Pi\Pi} & B_\Pi^\top \\
		B_\Delta^{(1)}  & \hdots &  B_\Delta^{(K)} & B_\Pi \\
	\end{pmatrix}
	\begin{pmatrix}
		\ul w_\Delta^{(1)} \\
		\vdots \\
		\ul w_\Delta^{(K)} \\
		\ul w_{\Pi} \\ \ul \lambda
	\end{pmatrix}
	=
	\begin{pmatrix}
		\ul f_\Delta^{(1)}  \\
		\vdots \\
		\ul f_\Delta^{(K)}  \\
		\ul f_\Pi \\ 0
	\end{pmatrix},
\end{equation}
where $S_{\Pi\Pi} :=\sum_{k=1}^K R^{(k)\top} S_{\Pi\Pi}^{(k)} R^{(k)}$ with
$S_{\Pi\Pi}^{(k)} :=A_{\Pi\Pi}^{(k)}-A_{\Pi \Delta}^{(k)}A_{\Delta\Delta}^{(k)-1} A_{\Delta\Pi}^{(k)}$,
and
$\ul f_\Pi  := - \sum_{k=1}^K	R^{(k)\top} A_{\Pi \Delta}^{(k)}A_{\Delta\Delta}^{(k)-1} \ul f_\Delta^{(k)} $,
and
$
		B_{\Pi}
		:=-
		\sum_{k=1}^K
			B_{\Delta}^{(k)}
			A_{\Delta\Delta}^{(k)-1}
			A_{\Delta\Pi}^{(k)}
		R^{(k)}.
$
The solution is then recovered from~\eqref{eq:3:kkt2} via
$
		\ul u_\Delta^{(k)}
		=
		\ul w_\Delta^{(k)}
		-
		A_{\Delta\Delta}^{(k)-1}A_{\Delta\Pi}^{(k)}R^{(k)} \ul w_\Pi
$.
In order to solve the overall problem, the Schur complement formulation
\[
		F \ul \lambda = \ul b
\]
with
\[
		F := \sum_{k=1}^K
		B_\Delta^{(k)}
		A_{\Delta\Delta}^{(k)-1}
		B_\Delta^{(k)\top}
		+ B_{\Pi} S_{\Pi\Pi}^{-1} B_{\Pi}^\top
\]
and
$
		\ul b :=
		\sum_{k=1}^K
		 B_\Delta^{(k)} A_{\Delta\Delta}^{(k)-1}\ul f_\Delta^{(k)}
		+ B_{\Pi} S_{\Pi\Pi}^{-1} f_{\Pi}
$
is solved with a preconditioned conjugate gradient solver.
Next, we observe that the constraints do not refer to the interior degrees of freedom (black circular dofs in Figure~\ref{fig:schema2}), therefore we subdivide the group ``$\Delta$'' further into interior degrees of freedom (``$I$'')  and interface degrees of freedom (``$\Gamma$'').
Using
\[
		A_{\Delta\Delta}^{(k)}
		=
		\begin{pmatrix}
			A_{II}^{(k)} & A_{I\Gamma}^{(k)}\\
			A_{\Gamma I}^{(k)} & A_{\Gamma\Gamma}^{(k)}\\
		\end{pmatrix},
		\qquad
		B_{\Delta}^{(k)}
		=
		\begin{pmatrix}
			0 & B_{\Gamma}^{(k)}\\
		\end{pmatrix},
\]
and
$S_{\Gamma\Gamma}^{(k)} := A_{\Gamma\Gamma}^{(k)}-A_{\Gamma I}^{(k)} A_{II}^{(k)-1}A_{I\Gamma}^{(k)}$, a block Gaussian elimination immediately yields
\begin{equation}\label{eq:3:Fdef}
		F = \sum_{k=1}^K
		B_\Gamma^{(k)}
		S_{\Gamma\Gamma}^{(k)-1} \\
		B_\Gamma^{(k)\top}
		+ B_{\Pi} S_{\Pi\Pi}^{-1} B_{\Pi}^\top.
\end{equation}
Motivated by this observation, we choose
\begin{equation}\label{eq:3:Mdef}
		M :=
		\frac{1}{4}
		\sum_{k=1}^K
		B_\Gamma^{(k)}
		S_{\Gamma\Gamma}^{(k)} \\
		B_\Gamma^{(k)\top}
\end{equation}
as preconditioner. Here, the factor $\tfrac 14$ corresponds to the idea of multiplicity scaling. In the next sections, we give an estimate for the condition number of $MF$. For the implementation of a conjugate gradient solver, only the ability of computing matrix vector products is required, the conjugate gradient method can be realized efficiently, like explained in Refs.~\refcite{ToselliWidlund:2005a} and~\refcite{SchneckenleitnerTakacs:2020}
for second order problems.

\section{Estimates for the discrete biharmonic extension}
\label{sec:4}

In this section, we compare the norms of different kinds of biharmonic extensions, i.e., functions that minimize the $|\cdot|_{H^2(\wh\Omega)}$-seminorm and that satisfy certain boundary conditions.
Analogously to Refs.~\refcite{SchneckenleitnerTakacs:2020} and \refcite{Nepomnyaschikh:1995}, we prove such results by explicitly giving extension operators.
For simplicity, we use the following notation.
\begin{notation}
	In what follows, $C$ is a generic positive constant that only depends on the constants  $C_G$ from~\eqref{eq:ass:nabla}, $C_T$ from~\eqref{eq:ass:trace}, $C_Q$ from~\eqref{eq:ass:quasi uniform}, and the diameter of~$\Omega$. We write $a\lesssim b$ if $a\le C\,b$ and $a\eqsim b$ if $a\lesssim b \lesssim a$.
\end{notation}

\subsection{Preliminary results}

We start with an approximation error estimate for an $H^2$-orthogonal projection that preserves the boundary data.

\begin{lemma}\label{lem:approxerror}
	Let $p\ge3$ and $\Xi$ be a $p$-open knot vector such that
	$S_{p,\Xi}\subset H^2(0,1)$.
	The  $H^2$-orthogonal projector $\Pi : H^2_0(0,1) \to S_{p,\Xi} \cap H^2_0(0,1)$ satisfies
	$| v|_{H^2(0,1)}^2  = | \Pi v|_{H^2(0,1)}^2  + | v-\Pi v|_{H^2(0,1)}^2$
	and the error estimate
	\[
		\| v-\Pi v\|_{L^2(0,1)}^2 \lesssim \wh h^2 | v|_{H^2(0,1)}^2 \quad\mbox{for all}\quad v\in H^2_0(0,1).
	\]
\end{lemma}
\begin{proof}
	We use the results from Ref.~\refcite{SandeManniSpeelers:2019}.
	A direct application is not possible since that paper does not consider the boundary conditions.
	Consider the scalar product
	$
		(v,w)_{H^2_{D}(0,1)}
		:=
		(v'',w'')_{L^2(0,1)}
		+
		v(0)w(0)
		+
		v'(0) w'(0).
	$
	For any $v \in H^2(0,1)$, let $v_h\in S_{p,\Xi}$ be such that
	$(v-v_h,w_h)_{H^2_D(0,1)}=0$
	for all $w_h\in S_{p,\Xi}$.
	Since the scalar product minimizes the $H^2$-seminorm, we have using the results
	from Ref.~\refcite{SandeManniSpeelers:2019}  that
	\begin{equation}\nonumber
		| v - v_h |_{H^2(0,1)}
		=
		\inf_{w_h \in S_{p,\Xi}}
		| v - w_h |_{H^2(0,1)}
		\lesssim
		\wh h^{2}
		| v |_{H^4(0,1)}
		\quad \mbox{for all} \quad v\in H^4(0,1).
	\end{equation}
	Let $v\in H^4(0,1)\cap H^2_0(0,1)$ arbitrary but fixed.
	By testing with $w_h(t)=1$, $w_h(t)=t$, $w_h(t)=t^2/2+t$,
	and $w_h(t)=-t^3/6+t^2/2+t+1$,
	we obtain $v_h(0)=v(0)=0$, $v_h'(0)=v'(0)=0$, $v_h'(1)=v'(1)=0$, and $v_h(1)=v(1)=0$.
	This shows $v_h\in S_{p,\Xi}\cap H^2_0(0,1)$ and thus
	\begin{equation}\label{eq:lem:approxerror}
		\inf_{w_h \in S_{p,\Xi} \cap H^2_0(0,1)}
		| v - w_h |_{H^2(0,1)}
		\lesssim
		\wh h^{2}
		| v |_{H^4(0,1)}
		\quad \mbox{for all} \quad v\in H^4(0,1)\cap H^2_0(0,1).
	\end{equation}
	A standard Aubin-Nitsche duality argument shows the desired result.
\end{proof}

For $p=2$, we cannot show the same result for the $H^2$-orthogonal projection. However, there is another projector that satisfies the same error estimate and which admits a stable decomposition with respect to the $H^2$ seminorm.
\begin{lemma}\label{lem:approxerror2}
	Let $p=2$ and $\Xi$ be a $p$-open knot vector with a maximum knot span of $\tfrac14$ (cf.~\eqref{eq:hlequarter}) such that $S_{p,\Xi}\subset H^2(0,1)$. Then, there is a projector $\Pi : H^2_0(0,1) \to S_{p,\Xi} \cap H^2_0(0,1)$
	satisfying
	$| \Pi v|_{H^2(0,1)}^2  + | v-\Pi v|_{H^2(0,1)}^2\lesssim | v|_{H^2(0,1)}^2$
	and the error estimate
	\begin{equation}\label{eq:lem:approxerror2}
		\| v-\Pi v\|_{L^2(0,1)}^2 \lesssim \wh h^2 | v|_{H^2(0,1)}^2 \quad\mbox{for all}\quad v\in H^2_0(0,1).
	\end{equation}
\end{lemma}
Since the proof of this Lemma requires estimates that will be introduced later, we give it only in the next subsection.
Moreover, we use a standard inverse estimate.

\begin{lemma}\label{lem:schwabinverse}
	Let $0\le s \le r \le p$ be integers and let $\Xi$ be a $p$-open knot vector such that $S_{p,\Xi}\subset H^r(0,1)$. Then, we have
	\[
		|v_h|_{H^r(0,1)} \le \left(2\sqrt{3} \, p^2\,\wh h_{\min}^{-1} \right)^{r-s} |v_h|_{H^s(0,1)}
		\quad\mbox{for all}\quad
		v_h \in S_{p,\Xi}.
	\]
\end{lemma}
\begin{proof}
	A proof is given in Theorem 3.91 of Ref.~\refcite{Schwab:1998} for the case of a polynomial of degree $p$ and for $r=1$ and $s=0$.
	An extension to larger values for $r$ and $s$ is straight forward since derivatives of polynomials are polynomials with smaller degree.
	The estimate for splines follows by taking a sum over all knot spans.
\end{proof}

\subsection{Extension operator for parameter domain}

Next, consider a tensor-product B-spline space
$\wh V := \mathrm{span}\,\Psi_{p,\mathbf \Xi}\subset H^2(\wh\Omega)$,
where $\mathbf \Xi = (\Xi_1,\Xi_2)=(\Xi_{k,1},\Xi_{k,2})$ is the knot vector for one of the patches $\Omega_k$.
Since we consider one fixed patch in this section, we omit  the index $k$ referring to the patch here and in the remainder of this subsection. The local function space is defined analogous to~\eqref{eq:Vkdef}.
Let
\begin{equation}\label{eq:vdelta def}
		\wh V_{\Delta}
		:=
		\{
			v \in \wh V :
			v(\wh\xx)=\partial_{\wh x} v(\wh \xx) = \partial_{\wh y} v(\wh \xx) = \partial_{\wh x\wh y}v(\wh \xx) = 0
			\mbox{ for all }
			\wh\xx \in \{0,1\}^2
		\}
\end{equation}
be the subspace of functions where the function values, the first derivatives and the mixed derivatives vanish on all corners.
In the following, we construct bounded extensions from $\wh \Gamma = \{0\}\times [0,1]$ into $\wh \Omega = (0,1)^2$.
Let
\[
		\wh V_0
			:=
		\{
			v \in \wh V_\Delta :
			v|_{\partial\wh\Omega \setminus \wh \Gamma} = 0,
			\partial_n v|_{\partial\wh\Omega \setminus \wh \Gamma} = 0
		\}
\]
be the functions that vanish on all interfaces but the interface $\wh \Gamma$.
Let
\[
	\wh W:=\{ v|_{\wh \Gamma} : v \in \wh V_\Delta \} =\{ \partial_n v |_{\wh \Gamma} : v \in \wh V_\Delta\} \subset H^2_0(\wh \Gamma)
\]
be the associated trace space for the function values and the normal derivatives.
Let
\begin{equation}\label{eq:big L def}
		L:=\lfloor \log_2 (\wh h^{-1}) \rfloor - 1;
\end{equation}
we know $L\ge 1$ because of~\eqref{eq:hlequarter}.
Recall that $\Xi_1$ is the knot vector of knots in $\wh x$-direction, i.e., perpendicular to $\wh \Gamma$.
For $\ell\in\{1,\dots,L\}$,  let
\begin{equation}\label{eq:xi eta def}
	\mu_\ell:= p 2^{\ell},
	\quad
	\xi_\ell := \max \{ \xi \in \Xi_1: \xi \le 2^{-\ell} \},
	\quad
	\eta_\ell := \max \{ \xi \in \Xi_1: \xi \le 2^{-\ell+1} \}.
\end{equation}
For $\alpha\in\{0,1\}$ and $\ell\in\{1,\dots,L\}$, define the extension operators via
\begin{equation}\label{eq:def simpleext}
		\wh E_{\alpha,\ell} : \wh W \to C^1(\wh\Omega)
		\quad\mbox{with}\quad
		(\wh E_{\alpha,\ell} w)(\wh x,\wh y) := \psi_{\alpha,\ell}(\wh x) w(\wh y)
		\quad \mbox{for all} \quad w\in \wh W,
\end{equation}
where
\begin{equation}\label{eq:def:psi}
\begin{aligned}
	\psi_{0,\ell}(x)
		& :=
			- \frac{\xi_\ell}{\eta_\ell-\xi_\ell} \max(0,1-x/\xi_\ell)^p
			+ \frac{\eta_\ell}{\eta_\ell-\xi_\ell} \max(0,1-x/\eta_\ell)^p,
		 \\
	\psi_{1,\ell}(x)
		& :=
			\frac{\xi_\ell \eta_\ell}{p(\eta_\ell-\xi_\ell)}  \max(0,1-x/\xi_\ell)^p
			- \frac{\xi_\ell \eta_\ell}{p(\eta_\ell-\xi_\ell)} \max(0,1-x/\eta_\ell)^p.
\end{aligned}
\end{equation}
First, we show that the operators are indeed extension operators.
\begin{lemma}\label{lem:41}
	For all $\ell\in\{1,\ldots,L\}$ and all $w\in \wh W$, we have
	\begin{align*}
		\wh E_{0,\ell} &: \wh W \to \wh V_0 \quad \mbox{with}\quad
		(\wh E_{0,\ell} w)|_{\wh\Gamma} = w, \quad
		\partial_n (\wh E_{0,\ell} w)|_{\wh\Gamma} = 0 , \\
		\wh E_{1,\ell} &: \wh W \to \wh V_0 \quad \mbox{with}\quad
		(\wh E_{1,\ell} w)|_{\wh\Gamma} = 0, \quad
		\partial_n (\wh E_{1,\ell} w)|_{\wh\Gamma} = w .
	\end{align*}
\end{lemma}
\begin{proof}
	Using~\eqref{eq:big L def}, \eqref{eq:xi eta def} and $\wh h \le 2^{-L-1}$, we know
	\begin{equation}\label{eq:xi eta facts}
		0 < 2^{-\ell-1}  < \xi_\ell \le 2^{-\ell} < \eta_\ell \le 2^{-\ell+1} \le 1
		\quad\mbox{and}\quad
		\eta_\ell-\xi_\ell \ge 2^{-\ell-1},
	\end{equation}
	and thus $0 < \xi_\ell < \eta_\ell < 1$.
	So, $\psi_{\alpha,\ell}\in \mathrm{span}\, \Psi_{p,\Xi_1}$ are splines which satisfy
	$\psi_{\alpha,\ell}(0)=\delta_{\alpha,0}$, $\psi_{\alpha,\ell}'(0)=\delta_{\alpha,1}$,
	$\psi_{\alpha,\ell}(1)=0$ and $\psi_{\alpha,\ell}'(1)=0$.
	From this observation,
	it immediately follows that $\wh E_{\alpha,\ell}w$ is a tensor-product B-spline and satisfies the desired boundary conditions on $\wh \Gamma=\{0\}\times[0,1]$ and the opposite side $\{1\}\times[0,1]$.
	The boundary conditions for the two remaining sides follow since the function values and the normal derivatives of functions in $\wh W \subset H^2_0(\wh \Gamma)$ vanish at $\partial\wh\Gamma=\{(0,0),(0,1)\}$.
\end{proof}
In order to estimate the norm of the extension operators, the following lemma is helpful.
\begin{lemma}\label{lem:l1 linf estim}
	For $\ell \in \{1,\ldots,L\}$, $\alpha\in\{0,1\}$ and $j\in\{0,1,2\}$, we have
	\begin{align*}
		\|\partial^j \psi_{\alpha,\ell}\|_{L^\infty(0,1)}
		& \le p^{j-\alpha} \frac{\eta_\ell^{1-j}\xi_\ell^\alpha+\xi_\ell^{1-j}\eta_\ell^\alpha}{\eta_\ell-\xi_\ell}
		\lesssim
		\,\mu_\ell^{j-\alpha},\\
		\|\partial^j \psi_{\alpha,\ell}\|_{L^1(0,1)}
		& \le
		p^{j-\alpha-1}
		\frac{\eta_\ell^{2-j}\xi_\ell^\alpha+\xi_\ell^{2-j}\eta_\ell^\alpha}{\eta_\ell-\xi_\ell}
		\lesssim
		\,\mu_\ell^{j-1-\alpha}.
	\end{align*}
\end{lemma}
\begin{proof}
	The first estimates of the $L^\infty$-norm uses that
	$\|\mathrm{max}(0,1-\cdot/t)^p\|_{L^\infty(0,1)}=1$ for $t\in\{\xi_\ell,\eta_\ell\}$,
	 the triangle inequality and the chain rule.
	The first estimates of the $L^1$-norm uses that
	$\|\mathrm{max}(0,1-\cdot/t)^p\|_{L^1(0,1)}=t/(1+p)\le t/p$ for $t\in\{\xi_\ell,\eta_\ell\}$,
	the triangle inequality and the chain rule. The second estimates follow using~\eqref{eq:xi eta facts} and the definition of $\mu_\ell$ in~\eqref{eq:xi eta def}.
\end{proof}
Next, we are able to give an estimate.
\begin{lemma}\label{lem:42}
	For all $\alpha \in\{0,1\}$ and all $w_1,\ldots,w_{L} \in \wh W$,
	we have
	\begin{align*}
		\left|\sum_{\ell=1}^{L} \wh E_{\alpha,\ell} w_\ell \right|_{H^2(\wh\Omega)}^2
		 &\lesssim
		\sum_{\ell=1}^{L}
		\mu_\ell^{-1-2\alpha}
		\left(
		\mu_\ell^4\|w_\ell \|_{L^2(\wh \Gamma)} ^2
		+
		|w_\ell |_{H^{2}(\wh \Gamma)} ^2
		\right) .
	\end{align*}
\end{lemma}
\begin{proof}
	By definition, we have
	\begin{equation}\label{eq:lem:42:1}
	\begin{aligned}
		\left|\sum_{\ell=1}^{L}\wh  E_{\alpha,\ell} w_\ell \right|_{H^2(\wh\Omega)}^2
		&  =
		\sum_{j=0}^2
		\sum_{\ell=1}^{L}
		\sum_{m=1}^{L}
		( \partial^j \psi_{\alpha,\ell},\partial^j \psi_{\alpha,m} )_{L^2(0,1)}
		( \partial_{\wh y}^{2-j} w_\ell,  \partial_{\wh y}^{2-j} w_m )_{L^2(\wh \Gamma)}.
	\end{aligned}
	\end{equation}
	By estimating the the first scalar product using the $L^\infty$ and the $L^1$ norm and those norms using Lemma~\ref{lem:l1 linf estim}, using the symmetry in $m$ and $\ell$ and~\eqref{eq:xi eta def}, we obtain
	\begin{align*}
		& ( \partial^j \psi_{\alpha,\ell},\partial^j \psi_{\alpha,m} )_{L^2(0,1)}
		\lesssim 2^{-|\ell-m|/2} \, \mu_\ell^{j-\alpha-1/2}\, \mu_m^{j-\alpha-1/2}.
	\end{align*}
	By inserting this into~\eqref{eq:lem:42:1}, using Cauchy-Schwarz and Young's inequalities
	and using symmetry in $\ell$ and $m$, we obtain
	\begin{align*}
		& \left|\sum_{\ell=1}^{L}\wh E_{\alpha,\ell} w_\ell \right|_{H^2(\wh\Omega)}^2\\
		&\quad  \lesssim
		\sum_{j=0}^2
		\sum_{\ell=1}^{L}
		\sum_{m=1}^{L}
		2^{-|\ell-m|/2}  \, (\mu_\ell^{2j-1-2\alpha}|w_\ell |_{H^{2-j}(\wh \Gamma)}^2+\mu_m^{2j-1-2\alpha}|w_m |_{H^{2-j}(\wh \Gamma)}^2) \\
		&\quad  \lesssim
		\sum_{j=0}^2
		\sum_{\ell=1}^{L}
		\sum_{m=1}^{L}
		2^{-|\ell-m|/2}
		 \, \mu_\ell^{2j-1-2\alpha}|w_\ell |_{H^{2-j}(\wh \Gamma)}^2.
	\end{align*}
	The summation formula for the geometric series further yields
	\begin{equation}\label{eq:lem:42:2}
		 \left|\sum_{\ell=1}^{L} \wh E_{\alpha,\ell} w_\ell \right|_{H^2(\wh\Omega)}^2
		\lesssim
		\sum_{j=0}^2
		\sum_{\ell=1}^{L}
		 \mu_\ell^{2j-1-2\alpha}|w_\ell |_{H^{2-j}(\wh \Gamma)} ^2.
	\end{equation}
	Integration by parts and Young's inequality yield
	$|w_\ell |_{H^1(\wh \Gamma)} ^2
	\le \| w_\ell \|_{L^2(\wh \Gamma)}  | w_\ell |_{H^2(\wh \Gamma)}
	\le \mu_\ell^2 \| w_\ell \|_{L^2(\wh \Gamma)} + \mu_\ell^{-2}  | w_\ell |_{H^2(\wh \Gamma)}^2 $.
	Substituting this into~\eqref{eq:lem:42:2} yields the desired result.
\end{proof}

Next, we define overall extension operators $\wh E_\alpha$ by explicitly prescribing a decomposition $w=w_1+\dots w_L$. The idea  is that the two summands in Lemma~\ref{lem:42} are supposed to be almost equal, so
$\mu_\ell^2 \|w_\ell\|_{L^2(\wh \Gamma)}\approx |w_\ell|_{H^{2}(\wh \Gamma)}$ is desired.
In order to precisely define the overall extension operator,
we first decompose $\wh W$ via an orthonormal eigendecomposition:
We consider the decomposition of $\wh W$ into functions
$\varphi_{1},\dots,\varphi_{N}\in \wh W$ (where $N=\mathrm{dim}\, \wh W$) with
\[
(\varphi_{i},\varphi_{j})_{L^2(\wh\Gamma)} = \delta_{i,j}, \quad
(\varphi_{i}'',\varphi_{j}'')_{L^2(\wh\Gamma)} = \lambda_{i}^2 \delta_{i,j}
\quad\mbox{and}\quad
0<\lambda_{1} \le \dots \le \lambda_{N}.
\]
Next, we collect the eigenvalues into buckets as follows:
\begin{equation}\label{eq:buckets}
\begin{aligned}
	\Lambda_1 & := \{ i : \lambda_i < \sqrt{\mu_1\mu_2} \},
		\\
	\Lambda_\ell & := \{ i : \sqrt{\mu_{\ell-1}\mu_{\ell}} \le \lambda_i < \sqrt{\mu_{\ell}\mu_{\ell+1}}  \}
		\qquad \mbox{for}\qquad \ell\in\{2,\ldots, L-1\}, \\
	\Lambda_{L} & := \{ i : \sqrt{\mu_{L-1}\mu_{L}}  \le \lambda_i  \}
\end{aligned}
\end{equation}
and define
\begin{equation}\label{eq:buckets2}
	\wh W_{\ell}
	:= \mathrm{span}\, \{ \varphi_{i} : i \in \Lambda_\ell \}.
\end{equation}
For each $w\in \wh W$, a decomposition of $w$ is defined via
\begin{equation}\label{eq:buckets3}
	w=\sum_{\ell=1}^{L} w_{\ell}
	\quad\mbox{with}\quad
	w_{\ell}:=\sum_{i \in \Lambda_\ell}
		(w,\varphi_{i})_{L^2(\wh\Gamma)}\varphi_{i}\in \wh W_\ell \quad\mbox{for}\quad \ell\in\{1,\ldots,L\}.
\end{equation}
Using this choice, we define the projector $\wh E_\alpha: \wh W \to \wh V_0$ for $\alpha\in\{0,1\}$ via
\begin{equation}\label{eq:Ealpha def}
		\wh E_\alpha w
		:=
		\sum_{\ell=1}^{L}
		\wh E_{\alpha, \ell}w_{\ell}
		=
		\sum_{\ell=1}^{L}
		\wh E_{\alpha, \ell}
		\left(
		\sum_{i \in \Lambda_\ell}
		(w,\varphi_{i})_{L^2(\wh\Gamma)}\varphi_{i}
		\right)
		.
\end{equation}
\begin{lemma}\label{lem:E is extension}
	For all $w\in\wh W$, we have
	\begin{align*}
		\wh E_{0} &: \wh W \to \wh V_0 \quad \mbox{with}\quad
		(\wh E_{0} w)|_{\wh\Gamma} = w, \quad
		\partial_n (\wh E_{0} w)|_{\wh\Gamma} = 0 , \\
		\wh E_{1} &: \wh W \to \wh V_0 \quad \mbox{with} \quad
		(\wh E_{1} w)|_{\wh\Gamma} = 0, \quad
		\partial_n (\wh E_{1} w)|_{\wh\Gamma} = w .
	\end{align*}
\end{lemma}
\begin{proof}
	This result follows from~\eqref{eq:Ealpha def},
	Lemma~\ref{lem:41},
	and~\eqref{eq:buckets3}.
\end{proof}
By using the bounds that we have already computed, we obtain the following bound for the norm of $\wh E_\alpha$.

\begin{lemma}\label{lem:E bound}
	For $\alpha\in\{0,1\}$ and all $w\in \wh W$, we have
	\[
		 |\wh E_\alpha w|_{H^2(\wh\Omega)}^2
		\lesssim
		\sum_{i=1}^{N}
		\min(\lambda_i,p\wh h^{-1})^{3-2\alpha}
			(w,\varphi_{i})_{L^2}^2
		+p^{3-2\alpha} \|w\|_{L^2}^2
		+ \left(p\,\wh h^{-1}\right)^{-1-2\alpha} |w|_{H^{2}}^2.
	\]
\end{lemma}
\begin{proof}
	Let $w\in \wh W$ arbitrary but fixed and $w_i:=(w,\varphi_{i})_{L^2(\wh\Gamma)}$.
	Using \eqref{eq:Ealpha def} and Lemma~\ref{lem:42}, we obtain
	\begin{equation}\label{eq:lem:E bound:1}
		|\wh E_\alpha w|_{H^2(\wh\Omega)}^2
		\lesssim
		\sum_{\ell=1}^{L}
		\sum_{i\in \Lambda_\ell}
		(
		\mu_\ell^{3-2\alpha}
		+
		\mu_\ell^{-1-2\alpha}
		\lambda_i^4
		)
		w_i^2.
	\end{equation}
	To estimate this, consider three cases.
	(i)
	For $\ell\in\{2,\ldots,L-1\}$ and $i\in\Lambda_\ell$, we have $\lambda_i \eqsim \mu_\ell \le \mu_{L} \eqsim p\wh h^{-1}$ due to~\eqref{eq:buckets}, \eqref{eq:xi eta def} and \eqref{eq:big L def}. This shows
	$\mu_\ell^{3-2\alpha}+	\mu_\ell^{-1-2\alpha}\lambda_i^4 \eqsim \lambda_i ^{3-2\alpha} \lesssim \min(\lambda_i ,p\wh h^{-1})^{3-2\alpha}$.
	(ii)
	For $\ell=1$ and $i\in\Lambda_\ell$, we have using~\eqref{eq:buckets} and \eqref{eq:xi eta def} that
	$\lambda_i \lesssim \mu_1 \eqsim p$ and thus
	$\mu_1^{3-2\alpha}
	+
	\mu_1^{-1-2\alpha}
	\lambda_i^4
	\lesssim
	\mu_1^{3-2\alpha}
	\eqsim
	p^{3-2\alpha}
	$.
	(iii)
	For $\ell=L$ and $i\in\Lambda_\ell$, we have using~\eqref{eq:buckets}, \eqref{eq:xi eta def} and \eqref{eq:big L def} that  $\lambda_i \gtrsim \mu_{L} \eqsim p\wh h^{-1}$ and thus
	$\mu_{L}^{3-2\alpha}
	+
	\mu_{L}^{-1-2\alpha}
	\lambda_i^4
	\lesssim
	(p\wh h^{-1})^{-1-2\alpha}
	\lambda_i^4
	$.
	By substituting these estimates into~\eqref{eq:lem:E bound:1}, we obtain
	\begin{equation}\nonumber
		|\wh E_\alpha w|_{H^2(\wh\Omega)}^2
		\lesssim
		\sum_{i=1}^{N}
		\big(
		\min(\lambda_i, p\wh h^{-1})^{3-2\alpha}
		+
		p^{3-2\alpha}
		+
		(p\wh h^{-1})^{-1-2\alpha}
		\lambda_i^4
		\big)
		w_i^2.
	\end{equation}
	The desired result follows immediately using Parseval's identity.
\end{proof}

Next, we give estimates of the norm of the extension operators using norms that are defined with Hilbert space interpolation techniques, specifically by norms defined via the K-method, see, e.g., Chapter~15 in Ref.~\refcite{Lions:1972}. Helpful information on Hilbert space interpolation can also be found in Ref.~\refcite{AdamsFournier:2003}.
Given two Hilbert spaces $X\subset Y$ where $X$ is dense in $Y$, we define for each $\theta \in (0,1)$ a new space $[X,Y]_{\theta}\subset Y$ via
\begin{equation}\label{eq:Kmethod1}
		[X,Y]_{\theta}
		:=
		\{
			y\in Y
			:
			\|y\|_{[X,Y]_{\theta}} < + \infty
		\},
\end{equation}
where
\begin{equation}\label{eq:Kmethod2}
		\|y\|_{[X,Y]_{\theta}}^2
		:=
		\|y\|_{Y}^2
		+
		\int_0^\infty
		\inf_{x\in X}
		\big(
		t^{-2\theta-1}
		\|x\|_X^2
		+
		t^{-2\theta+1}
		\|y-x\|_Y^2
		\big)
		\mathrm dt.
\end{equation}

The interpolation spaces can also be characterized using Fourier series. We show that the interpolation norms can also be represented based on the discrete eigendecomposition that we have introduced above.

\begin{lemma}\label{lem:47}
	For $\alpha\in\{0,1\}$ and all $w\in \wh W$, we have
	\[
		\sum_{i=1}^{N}
		\min(\lambda_i,\wh h^{-1})^{3-2\alpha}
			(w,\varphi_{i})_{L^2(\wh\Gamma)}^2
		\lesssim
		\|w\|^2_{[H^2_0(\wh\Gamma),L^2(\wh\Gamma)]_{(1+2\alpha)/4}}.
	\]
\end{lemma}
\begin{proof}
	Let $\alpha\in\{0,1\}$ and $w\in \wh W$ be arbitrary but fixed and $\theta := (1+2\alpha)/4$.
	Because of Lemmas~\ref{lem:approxerror} and \ref{lem:approxerror2}, we know that there is a projector $\wh \Pi:H^2_0(\wh\Gamma) \to \wh W$ and that there is a constant $1\le C_A\lesssim  1$ such that
	\[
			 \|v-\wh \Pi v\|_{L^2(\wh\Gamma)}^2
			\le C_A^2 \wh h^2 |v-\wh \Pi v|_{H^2(\wh\Gamma)}^2,
			\qquad
			|\wh \Pi v|_{H^2(\wh\Gamma)}^2
			+
			|v-\wh \Pi v|_{H^2(\wh\Gamma)}^2
			\le C_A^2 |v|_{H^2(\wh\Gamma)}^2
	\]
	for all $v\in H^2_0(\wh\Gamma)$.
	For  $t \le \sigma:=C_A^{-1} \wh h^{-2}$, we have using these estimates and
	the triangle inequality
	\begin{align*}
		\|w\|^2_{[H^2_0(\wh\Gamma),L^2(\wh\Gamma)]_{\theta}}
		& \ge \int_0^{\sigma}
		\inf_{v \in H^2_0(\wh\Gamma)}
		\big(
		 t^{-2\theta-1} |v|_{H^2(\wh\Gamma)}^2
		+ t^{-2\theta+1} \|w-v\|_{L^2(\wh\Gamma)}^2
		\\&\qquad\qquad
		- C_A^{-2} t^{-2\theta-1}
		( |v-\wh\Pi v|_{H^2(\wh\Gamma)}^2
			- t^{2} \|v-\wh\Pi v\|_{L^2(0,1)}^2  )
		\big)
		\,\mathrm dt
		\\&
		\gtrsim C_A^{-2} \int_0^{\sigma}
		\inf_{v \in H^2_0(\wh\Gamma)}
		\big(
		 t^{-2\theta-1} |\wh\Pi v|_{H^2(\wh\Gamma)}^2
		+ t^{-2\theta+1} \|w-\wh\Pi v\|_{L^2(\wh\Gamma)}^2
		\big)
		\,\mathrm dt.
	\end{align*}
	Since $\wh\Pi$ maps into $\wh W$, we further have
	\begin{align*}
		\|w\|^2_{[H^2_0(\wh\Gamma),L^2(\wh\Gamma)]_{\theta}}
		\gtrsim \int_0^{\sigma}
		\inf_{v \in \wh W}
		\big(
		t^{-2\theta-1} |v|_{H^2(\wh\Gamma)}^2
		+ t^{-2\theta+1} \|w-v\|_{L^2(\wh\Gamma)}^2
		\big)
		\,\mathrm dt.
	\end{align*}
	Analogous to $w= \sum_{i=1}^{N} w_i\varphi_{i}$, also each
	$v \in \wh W$ can be uniquely written as
	$v= \sum_{i=1}^{N} v_i \varphi_{i}$.
	Parseval's identity yields
	$|v|_{H^2(\wh\Gamma)}^2 = \sum_{i=1}^{N}  \lambda_i^4 v_i^2$
	and
	$\|w-v\|_{L^2(\wh\Gamma)}^2 = \sum_{i=1}^{N} (w_i-v_i)^2$.
	Using this, we explicitly compute the infimum and have
	\begin{align*}
		\|w\|^2_{[H^2_0(\wh\Gamma),L^2(\wh\Gamma)]_{\theta}}
		&
		\gtrsim
		\int_0^{\sigma}
		\sum_{i=1}^{N}
		\inf_{v_i \in \mathbb R}
		\big(
		t^{-2\theta-1} \lambda_i^4 v_i^2
		+ t^{-2\theta+1} (w_i-v_i)^2
		\big)
		\,\mathrm dt\\
		&
		=
		\sum_{i=1}^N
		\int_0^{\sigma}
		\frac{t^{1-2\theta}  \lambda_i^4}{t^2+ \lambda_i^4}
		\,\mathrm dt \,w_i^2
		\gtrsim
		\sum_{i=1}^{N}
		\min(\sigma,\lambda_i^2)^{2-2\theta}
		w_i^2.
	\end{align*}
	Since $\sigma\eqsim \wh h^{-2}$, this implies the desired result.
\end{proof}

The interpolation space from the last lemma can be represented as a classical Sobolev space with fractional order as defined with the the Sobolev-Slobodeckij formulation. The norms for the relevant spaces are given by
\[
	|w|_{H^{1/2}(\wh\Gamma)}^2 := \int_{\wh\Gamma}\int_{\wh\Gamma}
		\frac{ (w(\wh{\mathbf x}) - w(\wh{\mathbf y}))^2 }{ \| \wh{\mathbf x} - \wh{\mathbf y} \|_{\ell^2}^2}
		\, \mathrm d \wh{\mathbf x}
		\, \mathrm d \wh{\mathbf y},
\]
and
$	|w|_{H^{n+1/2}(\wh\Gamma)}^2 :=|\partial^n w|_{H^{1/2}(\wh\Gamma)}^2$
and
$	\|w\|_{H^{n+1/2}(\wh\Gamma)}^2 :=|w|_{H^{n+1/2}(\wh\Gamma)}^2 + \|w\|_{H^n(\wh\Gamma)}^2$
for $n \in \mathbb N_0 = \{0,1,2,3,\dots\}$.
Moreover, define for $n\in \mathbb N_0$
\[
	\| v \|_{H^{n+1/2}_{00}(\wh \Gamma)}^2
			:=
			\| v \|_{H^{n+1/2}(\wh \Gamma)}^2
			+
			\| \delta^{-1/2} \partial^n v \|_{L^2(\wh \Gamma)}^2
	\quad\mbox{with}\quad
	\delta(\wh x,\wh y) := \wh y(1-\wh y).
\]
The function spaces $H^{n+1/2}(\wh\Gamma)$  and $H_{00}^{n+1/2}(\wh\Gamma)$ contain all functions in $L^2(\wh\Gamma)$ with finite $H^{n+1/2}$ or $H^{n+1/2}_{00}$ norm, respectively.
The following result is standard.

\begin{lemma}\label{lem:h00}
	Let $s_1>s_2\ge 0$, $\theta \in (0,1)$, and $s:=(1-\theta) s_1 + \theta s_2$.
	Then, we have $[H^{s_1}(\wh \Gamma), H^{s_2}(\wh \Gamma)]_\theta = H^s(\wh \Gamma)$ and there is a constant $C_{s_1,s_2,\theta}\ge 1$ such that
	\begin{align*}
			C_{s_1,s_2,\theta}^{-1}
			\| v \|_{H^s(\wh \Gamma)}
			&\le
			\| v \|_{[H^{s_1}(\wh \Gamma), H^{s_2}(\wh \Gamma)]_\theta}
			\le
			C_{s_1,s_2,\theta}
			\| v \|_{H^s(\wh \Gamma)}
			\quad \mbox{for all} \quad v\in H^s(\wh \Gamma).
	\end{align*}
	If also  $s \in \mathbb N_0 + \tfrac12$, we have
	$[H_0^{s_1}(\wh \Gamma), H_0^{s_2}(\wh \Gamma)]_\theta
	=
	H^s_{00}(\wh \Gamma)$, again, there is a constant $C_{s_1,s_2,\theta}\ge 1$ such that
	\begin{align*}
			C_{s_1,s_2,\theta}^{-1}
			\| v \|_{H^s_{00}(\wh \Gamma)}
			&\le
			\| v \|_{[H_0^{s_1}(\wh \Gamma), H_0^{s_2}(\wh \Gamma)]_\theta}
			\le
			C_{s_1,s_2,\theta}
			\| v \|_{H^s_{00}(\wh \Gamma)}
			\quad \mbox{for all} \quad v\in H^s_{00}(\wh \Gamma).
	\end{align*}
\end{lemma}
\begin{proof}
      The combination of Theorems~9.6 and~15.1 in Ref.~\refcite{Lions:1972} states $[H_0^{s_1}(\wh \Gamma), H_0^{s_2}(\wh \Gamma)]_\theta=H^s(\wh \Gamma)$ with equivalent norms.
      The combination of Theorems~11.7 and~15.1 in Ref.~\refcite{Lions:1972} gives the characterization of $[H_0^{s_1}(\wh \Gamma), H_0^{s_2}(\wh \Gamma)]_\theta$ via  $v\in H^s(\wh \Gamma)$ and $\delta^{-1/2} \partial^{n} v \in L^2(\wh \Gamma)$ for $s=n+1/2$ since the function $\delta$ satisfies the conditions of (11.16) in Ref.~\refcite{Lions:1972}. Again, that theorem states that the norms are equivalent. The constant used for interpolation can depend on all involved quantities, except the function $v$ itself. Since we have fixed the domain and $\delta$, the only remaining independent quantities are $s_1$, $s_2$ and $\theta$.
\end{proof}
Since $H^0_0(\wh \Gamma)=H^0(\wh \Gamma)=L^2(\wh \Gamma)$, we have
$[H_0^2(\wh \Gamma), L^2(\wh \Gamma)]_{(1+2\alpha)/4} = H^{1/2+\alpha}_{00}(\wh \Gamma)$
and
$[H^2(\wh \Gamma), L^2(\wh \Gamma)]_{(1+2\alpha)/4} = H^{1/2+\alpha}(\wh \Gamma)$
for $\alpha\in\{0,1\}$, in both cases with equivalent norms.
Next we extend the inverse inequality to fractional order spaces.
\begin{lemma}\label{lem:interpolinverse}
	For $r\in \{1,2\}$ and $s\in\{1/2,3/2\}$, $s\le r$, and all $w\in \wh W$, the estimate
	$
	\|w\|_{H^r(\wh\Gamma)}^2
	\lesssim (p^4 \wh h^{-2})^{r-s} \|w\|_{H^s(\wh\Gamma)}^2
	$
	holds.
\end{lemma}
\begin{proof}
	Using the reiteration theorem (Theorem~6.1 in Ref.~\refcite{Lions:1972})
	and Lemma~\ref{lem:h00},
	we have
	$
			\|w\|_{H^r(\wh \Gamma)}
			\eqsim
			\|w\|_{[H^s(\wh\Gamma),H^{r+1}(\wh \Gamma)]_\theta}
	$
	with $\theta = (r-s)/(r-s+1)$.
	Using the monotony of the interpolation (Proposition~2.3 in Ref.~\refcite{Lions:1972})
	and Lemma~\ref{lem:schwabinverse},
	we further obtain
	$
			\|w\|_{H^r(\wh \Gamma)}
			\lesssim
			\|w\|_{H^{r+1}(\wh \Gamma)}^\theta
			\|w\|_{H^s(\wh\Gamma)}^{1-\theta}
			\lesssim
			(\wh h^{-1} p^2)^\theta \|w\|_{H^r(\wh \Gamma)}^\theta
			\|w\|_{H^s(\wh\Gamma)}^{1-\theta},
	$
	which immediately implies the desired result.
\end{proof}

\begin{lemma}\label{lem:h00 estim}
	For all $w\in \wh W$, we have
	\begin{align*}
	\|w\|_{H_{00}^{3/2}(\wh \Gamma)}^2
	&\lesssim
		| w' |_{H^{1/2}(\wh \Gamma)}^2
		+
		(1+\log p + \log \wh h^{-1})
		\|  w ' \|_{L^{\infty}(\wh \Gamma)}^2,
	\\
	\|w\|_{H_{00}^{1/2}(\wh \Gamma)}^2
	&\lesssim
		| w |_{H^{1/2}(\wh \Gamma)}^2
		+
		(1+\log p + \log \wh h^{-1})
		\|w\|_{L^{\infty}(\wh \Gamma)}^2.
	\end{align*}
\end{lemma}
\begin{proof}
		Let $w\in \wh W$ arbitrary but fixed.
		For any $\epsilon \in (0,1/2)$, we have using the fundamental lemma of calculus
		and because of $w(0)=0$ that
		\begin{align*}
			\int_0^\epsilon \frac{w^2(t)}{t(1-t)}  \mathrm dt
			&= \int_0^\epsilon \frac{1}{t(1-t)} \left(\int_0^t  w'(t) \,\mathrm dx\right)^2 \mathrm dt
			\lesssim\epsilon |w|_{H^1(0,1/2)}^2.
		\end{align*}
		Using $w(1)=0$, we also obtain
		$\int_{1-\epsilon}^1 \frac{w^2(t)}{t(1-t)}  \mathrm dt
		\lesssim\epsilon |w|_{H^1(1/2,1)}^2$.
		Additionally, we have
		\begin{align*}
			\int_\epsilon^{1-\epsilon} \frac{w^2(t)}{t(1-t)}  \mathrm dt
			& \le
			\int_\epsilon^{1-\epsilon} \frac{1}{t(1-t)} \mathrm dt \|w\|_{L^\infty(0,1)}^2
			\lesssim   \log(\epsilon^{-1})\|w\|_{L^\infty(0,1)}^2.
		\end{align*}
		By combining these results, we obtain
		\begin{align}\label{eq:bdyint}
			\int_0^1 \frac{w^2(t)}{t(1-t)}  \mathrm dt
			& \le
			\epsilon |w|_{H^1(0,1)}^2
			+  \log(\epsilon^{-1}) \|w\|_{L^\infty(0,1)}^2.
		\end{align}
		Using
		Lemma~\ref{lem:h00},
		\eqref{eq:bdyint},
		Lemma~\ref{lem:interpolinverse},
		the choice $\epsilon := p^{-4} \wh h^{2}$, and
		Theorem~8.3 in Ref.~\refcite{Lions:1972}, we obtain
		\begin{equation}\label{eq:thrm42:eq2}
		\begin{aligned}
			\|w\|_{H^{1/2}_{00}(\wh \Gamma)}^2
			&
			=
			\|w\|_{H^{1/2}(\wh \Gamma)}^2
			+
			\int_0^1 \frac{w^2(t)}{t(1-t)}  \mathrm dt
			\\
			&
			\lesssim
			\|w\|_{H^{1/2}(\wh \Gamma)}^2
			+
			\epsilon
			|w|_{H^{1}(\wh \Gamma)}^2
			+
			\log(\epsilon^{-1})
			\|w\|_{L^{\infty}(\wh \Gamma)}^2
			\\
			&
			\lesssim
			\|w\|_{H^{1/2}(\wh \Gamma)}^2
			+
			 (1+\log p + \log \wh h^{-1})
			\|w\|_{L^{\infty}(\wh \Gamma)}^2.
		\end{aligned}
		\end{equation}
		Since $\|w\|_{H^{1/2}(\wh \Gamma)}^2=
		\|w\|_{L^{2}(\wh \Gamma)}^2+ |w|_{H^{1/2}(\wh \Gamma)}^2$
		and $\|w\|_{L^{2}(\wh \Gamma)}^2\le \|w\|_{L^{\infty}(\wh \Gamma)}^2$, the desired estimate immediately follows.
		Using the same arguments, we also obtain for every $w\in \wh W$
		\begin{equation}\label{eq:thrm42:eq3}
		\begin{aligned}
			\|w\|_{H^{3/2}_{00}(\wh \Gamma)}^2
			&
			\lesssim
			\|w\|_{H^{3/2}(\wh \Gamma)}^2
			+
			(1+\log p + \log \wh h^{-1})
			\| w'\|_{L^{\infty}(\wh \Gamma)}^2.
		\end{aligned}
		\end{equation}
		Again, we use
		$\|w\|_{H^{3/2}(\wh \Gamma)}^2=
		\|w\|_{L^{2}(\wh \Gamma)}^2 +
		|w|_{H^{1}(\wh \Gamma)}^2 +
		|w'|_{H^{1/2}(\wh \Gamma)}^2$.
		Since $w\in \wh W$ vanishes at both ends of $\wh \Gamma$, a standard Poincaré-Friedrichs inequality, cf. Theorem~A.18 in Ref.~\refcite{ToselliWidlund:2005a}, yields $\|w\|_{L^{2}(\wh \Gamma)}^2 \lesssim |w|_{H^{1}(\wh \Gamma)}^2$. Since we also have
		$|w|_{H^{1}(\wh \Gamma)}^2\le \| w'\|_{L^{\infty}(\wh \Gamma)}^2$,		this finishes the proof.
\end{proof}

By combining these results, we obtain as follows.
\begin{theorem}\label{thrm:ealpha bound0}
	For all $w\in \wh W$, we have
	\begin{align*}
		\|\wh E_0 w\|_{H^2(\wh \Omega)}^2
		&\lesssim
		p^3
		|w'|_{H^{1/2}(\wh\Gamma)}^2
		+
		p^3
		(1+\log p + \log \wh h^{-1}) \| w' \|_{L^{\infty}(\wh \Gamma)}^2,
		\\
		\|\wh E_1 w\|_{H^2(\wh \Omega)}^2
		&\lesssim
		p^3
		|w|_{H^{1/2}(\wh\Gamma)}^2
		+
		p^3
		(1+\log p + \log \wh h^{-1}) \| w \|_{L^{\infty}(\wh \Gamma)}^2.
	\end{align*}
\end{theorem}
\begin{proof}
	Let $\alpha\in\{0,1\}$.
	By combining
	Lemmas~\ref{lem:E bound},
	\ref{lem:47}, and
	\ref{lem:interpolinverse} and the monotonicity result
	$\|\cdot\|_{L^2(\wh\Gamma)} \lesssim \|\cdot\|_{[H^2_0(\wh\Gamma),L^2(\wh\Gamma)]_{(1+2\alpha)/4}}$, we obtain
	\[
		|\wh  E_\alpha w|_{H^2(\wh\Omega)}^2\lesssim p^{3-2\alpha}	\|w\|^2_{[H^2_0(\wh\Gamma),L^2(\wh\Gamma)]_{(1+2\alpha)/4}}.
	\]
	Since $E_\alpha w$ and its derivatives vanish on $\partial\wh\Omega\setminus\wh\Gamma$, one obtains
	$\|\wh E_\alpha w\|_{L^2(\wh\Omega)}^2
	\lesssim |\wh E_\alpha w|_{H^1(\wh\Omega)}^2$
	using a standard Poincaré-Friedrichs inequality, cf. Theorem~A.18 in Ref.~\refcite{ToselliWidlund:2005a}, and by applying such inequality to $\nabla\wh E_\alpha w$ also
	$|\wh E_\alpha w|_{H^1(\wh\Omega)}^2
	\lesssim |\wh E_\alpha w|_{H^2(\wh\Omega)}^2$.
	Combining these results yields
	\[
		\|\wh E_\alpha w\|_{H^2(\wh\Omega)}^2
		\lesssim p^{3-2\alpha}
		\|w\|^2_{[H^2_0(\wh\Gamma),L^2(\wh\Gamma)]_{(1+2\alpha)/4}}.
	\]
	By combining this with
	Lemmas~\ref{lem:h00} and
	\ref{lem:h00 estim}, we obtain the desired estimate.
\end{proof}

The extension operators $\wh E_\alpha$ extend functions that are defined on the edge $\wh \Gamma$ to the interior of the patch, i.e., $\wh \Omega$. For second order problems, such kind of extension operators were sufficient in order to analyze IETI-DP solvers, cf. the proofs in Ref.~\refcite{SchneckenleitnerTakacs:2020}. For fourth order problems this is not as straight forward. In order to be able to define an extension operator for the physical patch later, we introduce an extension operator $\wh E: \wh V_\Delta \to \wh V_\Delta$ that maps any function $\wh u  \in \wh V_\Delta$ to another function $\wh v  \in \wh V_\Delta$ whose function values and derivatives coincide with those of $\wh u$ and such that the norm of $\wh v$ can be  bounded by terms that only depend on its values and derivatives on $\wh \Gamma$.
So, we define
\begin{equation}\label{eq:whEdef}
		\wh E: \wh V_\Delta \to \wh V_\Delta
		\quad\mbox{via}\quad
		\wh E \,\wh u
		:= \wh E_0 ( \wh u|_{\wh \Gamma})
		+ \wh E_1 ( \partial_{n} \wh u|_{\wh \Gamma})
		\quad\mbox{for all}\quad \wh u \in \wh V_\Delta.
\end{equation}
We immediately obtain as follows.
\begin{theorem}\label{thrm:ealpha bound}
	The statements
	\[
		(\wh E \,\wh u)|_{\wh \Gamma} = \,\wh u|_{\wh \Gamma},
		\quad
		\nabla (\wh E \,\wh u)|_{\wh \Gamma} = \nabla \wh u|_{\wh \Gamma},
\quad
		(\wh E \,\wh u)|_{\partial\wh\Omega\setminus\wh \Gamma} = 0,
		\quad
		\nabla (\wh E \,\wh u)|_{\partial\wh\Omega\setminus\wh \Gamma} = 0
	\]
	and the estimate
	\[
		\|\wh E \,\wh u \|_{H^2(\wh \Omega)}^2
		\lesssim
		p^3
		|\nabla \wh u |_{H^{1/2}(\wh\Gamma)}^2
		+
		p^3
		(1+\log p + \log \wh h^{-1})
		\| \nabla \wh u \|_{L^{\infty}(\wh \Gamma)}^2
	\]
	hold for all $\wh u \in  \wh V_\Delta$.
\end{theorem}
\begin{proof}
	Let $\wh u \in \wh V_\Delta$ be arbitrary but fixed
	and let $C_\Lambda:=1+\log p + \log \wh h^{-1}$.
	The first four statements are an immediate consequence of~\eqref{eq:whEdef}
	and Lemma~\ref{lem:E is extension}.
	Using \eqref{eq:whEdef}, the triangle inequality and Theorem~\ref{thrm:ealpha bound0}, we obtain
	\begin{align*}
		&\|\wh E \,\wh u \|_{H^2(\wh \Omega)}^2
		\lesssim
		\|\wh E_0 \,\wh u\|_{H^2(\wh \Omega)}^2
		+
		\|\wh E_1 \partial_{\wh x} \,\wh u\|_{H^2(\wh \Omega)}^2
		\\
		&\qquad\lesssim
		p^3
		|\partial_{\wh y} \wh u|_{H^{1/2}(\wh\Gamma)}^2
		+
		p^3C_\Lambda \| \partial_{\wh y} \wh u \|_{L^{\infty}(\wh \Gamma)}^2
		+
		p^3
		|\partial_{\wh x} \wh u|_{H^{1/2}(\wh\Gamma)}^2
		+
		p^3C_\Lambda \| \partial_{\wh x} \wh u \|_{L^{\infty}(\wh \Gamma)}^2
		\\
		&\qquad=
		p^3|\nabla \wh u|_{H^{1/2}(\wh\Gamma)}^2
		+
		p^3C_\Lambda \| \nabla \wh u \|_{L^{\infty}(\wh \Gamma)}^2,
	\end{align*}
	i.e., the desired estimate.
\end{proof}

For completeness, we provide a proof of Lemma~\ref{lem:approxerror2}.
\begin{proofof}{Lemma~\ref{lem:approxerror2}}
	We choose $\Pi: H^2_0(0,1) \to S_{2,\Xi} \cap H^2_0(0,1)$ to be the $L^2$-orthogonal projection.
	We define two auxiliary projectors.
	$\Pi_1: H^1_0(0,1) \to S_{2,\Xi} \cap H^1_0(0,1)$ is the $H^1$-orthogonal projection.
	Completely analogous to Lemma~\ref{lem:approxerror}, one can show $\|v-\Pi_1v\|_{L^2(0,1)} \lesssim \wh h^2 |v|_{H^2(0,1)}$ for all $v\in H^1_0(0,1)$. (For the equidistant case, this was proven in Ref.~\refcite{Takacs:2018}.)
	Let $\Pi_2: H^2_0(0,1) \to S_{3,\widetilde\Xi} \cap H^2_0(0,1)$ be the $H^2$-orthogonal projection, where
	 $\widetilde\Xi =(0,\Xi,1)$ is the corresponding $3$-open knot vector.
	Lemma~\ref{lem:approxerror} states $\|v-\Pi_2v\|_{L^2(0,1)} \lesssim \wh h^2 |v|_{H^2(0,1)}$ for all $v\in H^2_0(0,1)$.
	For arbitrary $u\in H^2_0(0,1)$, let $u_h(t):=\Pi_1 u(t) - \psi_{1,L}(t)(\Pi_1 u)'(0) + \psi_{1,L}(1-t)(\Pi_1 u)'(1)$, where $\psi_{1,L}$ is as defined in~\eqref{eq:def:psi} and observe $u_h\in S_{2,\Xi}\cap H^2_0(0,1)$.
	Straight-forward computations yield $\|\psi_{1,L}\|_{L^2(0,1)}^2 \lesssim \wh h^3$.
	Using this, the fundamental theorem of calculus, Lemma~\ref{lem:schwabinverse} and the abovementioned error estimates, we have
	$\|u-\Pi u\|_{L^2(0,1)}
		\le \|u-u_h\|_{L^2(0,1)}
		\le \|u-\Pi_1 u\|_{L^2(0,1)}
			+ \|\psi_{1,L}\|_{L^2(0,1)} ( |(\Pi_1u-\Pi_2u)'(0)|+ |(\Pi_1u-\Pi_2u)'(1)|)
		\lesssim \|u-\Pi_1 u\|_{L^2(0,1)}
			+ \wh h^{3/2} \|\Pi_1u-\Pi_2u\|_{H^1(0,1)}^{1/2}\|\Pi_1u-\Pi_2u\|_{H^2(0,1)}^{1/2}
		\lesssim \|u-\Pi_1 u\|_{L^2(0,1)} + \|\Pi_1u-\Pi_2u\|_{L^2(0,1)}
		\lesssim \|u-\Pi_1 u\|_{L^2(0,1)} + \|u-\Pi_2u\|_{L^2(0,1)}
		\lesssim \wh h^2 |u|_{H^2(0,1)}$,
	i.e., the desired error estimate.
	Completely analogously, we have
	$|\Pi u|_{H^2(0,1)} + |u-\Pi u|_{H^2(0,1)} \lesssim |u|_{H^2(0,1)} + |\Pi_2 u|_{H^2(0,1)}+|\Pi  u- \Pi_2 u|_{H^2(0,1)}
	\lesssim |u|_{H^2(0,1)} + \wh h^{-2} \|\Pi  u- \Pi_2 u\|_{L^2(0,1)}
	\lesssim |u|_{H^2(0,1)}$, i.e., the desired stability result.
\end{proofof}

\subsection{Extension operator on the physical domain}

In this subsection, we denote the patch index explicitly, so $\wh V_k$ and $\wh V_{\Delta}^{(k)}$ denote the spaces defined in~\eqref{eq:Vkdef} and \eqref{eq:vdelta def}, respectively. Analogous to the definition of $V_k$, we also define
$
		V_{\Delta}^{(k)}
		:=
		\{
				v : v\circ G_k \in \wh V_{\Delta}^{(k)}
		\}
$.
We define the extension operator $E^{(k,\ell)}: V_\Delta^{(k)} \to V_\Delta^{(k)}$ using the pull-pack principle. Provided that the pre-image of $\Gamma_{k,\ell}$ is $\wh \Gamma$, we have
\[
			E^{(k,\ell)} w = (\wh E (w \circ G_k))\circ G_k^{-1}.
\]
If the pre-image of $\Gamma_{k,\ell}$ is one of the other three sides of $\wh\Omega$, the definition is analogous.
Before giving an estimate of the norm of $E^{(k,\ell)} w$, we give two Lemmas that guarantee the equivalence of relevant norms.

\begin{lemma}\label{lem:geoequiv1}
	Let $\Omega_k$ be one of the patches,
	$u\in H^2(\Omega_k)$ and $\wh u := u \circ G_k$ be its pullback to $\wh \Omega$.
	Then, $|u|_{H^1(\Omega_k)}^2 \eqsim |\wh u|_{H^1(\wh \Omega)}^2$ and $|u|_{H^2(\Omega_k)}^2 \lesssim H_k^{-2} \|\wh u\|_{H^2(\wh \Omega)}^2$ hold.
	Moreover, for each side $\Gamma_{k,\ell}$ with pull-back $\wh\Gamma_{k,\ell}:=G_k^{-1}(\Gamma_{k,\ell})$, the estimates
	$\|\nabla \wh u\|_{L^{2}(\wh \Gamma_{k,\ell})}^2
				\eqsim H_k  \|\nabla u\|_{L^{2}(\Gamma_{k,\ell})}^2$
	and $\|\nabla \wh u\|_{L^\infty(\wh\Gamma_{k,\ell})}^2
				\eqsim H_k^2 \|\nabla u\|_{L^\infty(\Gamma_{k,\ell})}^2$
	hold.
\end{lemma}
\begin{proof}
	The results follow using standard chain an substitution rules and~\eqref{eq:ass:nabla}.
	Special care has to be taken only for the estimate of the $H^2$-seminorm.
	For each fixed $u\in H^2(\Omega_k)$, let the function $\widetilde u$ be given by $\widetilde u(\xx) = u(H_k \xx)$ and set $\wh u := u \circ G_k$.
	Using~\eqref{eq:ass:nabla:tilde} and standard chain and product rules for differentiation and the substitution rule for the integral, we obtain
	$H_k^{2} |u|_{H^2(\Omega_k)}^2= |\widetilde u|_{H^2(H_k^{-1} \Omega_k)}^2
	\lesssim |\wh u|_{H^2(\wh \Omega)}^2
	+ |\wh u|_{H^1(\wh \Omega)}^2
	\le  \|\wh u\|_{H^2(\wh \Omega)}^2$.
	The desired estimate follows immediately.
\end{proof}
\begin{lemma}\label{lem:geoequiv2}
	Let $\Omega_k$ be one of the patches, $\Gamma_{k,\ell}$ one of it sides with pull-back $\wh\Gamma_{k,\ell}:=G_k^{-1}(\Gamma_{k,\ell})$,
	$u\in H^2(\Omega_k)$ and $\wh u := u \circ G_k$.
	Then, the estimate
	$|\nabla \wh u|_{H^{1/2}(\wh \Gamma_{k,\ell})}^2
				\lesssim H_k^2  |\nabla u |_{H^{1/2}(\Gamma_{k,\ell})}^2
					+  H_k \|\nabla u\|_{L^2(\Gamma_{k,\ell})}^2$
	holds.
\end{lemma}
\begin{proof}
	First, we observe
	$\nabla \wh u(\wh{\mathbf x}) = \nabla G(\wh{\mathbf x})\nabla u(\mathbf x)$; here and in what follows, we use the notation $\mathbf x = G_k(\wh{\mathbf x})$ and $\mathbf y = G_k(\wh{\mathbf y})$.
	Using the triangle inequality, we have
	\begin{align*}
			|\nabla \wh u|_{H^{1/2}(\wh \Gamma_{k,\ell})}^2
			 &=\int_{\wh \Gamma_{k,\ell}}\int_{\wh \Gamma_{k,\ell}}
			\frac{\|\nabla G_k(\wh{\mathbf x}) \nabla u(\mathbf x)-\nabla G_k(\wh{\mathbf y}) \nabla u(\mathbf y)\|_{\ell^2}^2}{\|\wh{\mathbf x}-\wh{\mathbf y}\|_{\ell^2}^2} \, \mathrm d\wh{\mathbf x} \, \mathrm d \wh{\mathbf y} \\
			&\le \int_{\wh \Gamma_{k,\ell}}\int_{\wh \Gamma_{k,\ell}}
			\frac{\|\nabla G_k(\wh{\mathbf x})\|_{\ell^2}\| \nabla u({\mathbf x})- \nabla u({\mathbf y})\|_{\ell^2}^2}{\|\wh{\mathbf x}-\wh{\mathbf y}\|_{\ell^2}^2} \, \mathrm d\wh{\mathbf x} \, \mathrm d \wh{\mathbf y} \\
			&\qquad + \int_{\wh \Gamma_{k,\ell}}\int_{\wh \Gamma_{k,\ell}}
			\frac{\|\nabla G_k(\wh{\mathbf x}) -\nabla G_k(\wh{\mathbf y})\|_{\ell^2}^2\| \nabla u({\mathbf y})\|_{\ell^2}^2}{\|\wh{\mathbf x}-\wh{\mathbf y}\|_{\ell^2}^2} \, \mathrm d\wh{\mathbf x} \, \mathrm d \wh{\mathbf y},
	\end{align*}
	where $\|\cdot\|_{\ell^2}$ applied to matrices is the spectral norm.
	Using the fundamental theorem of calculus, we have
	$\|\nabla G(\wh{\mathbf x}) - \nabla G(\wh{\mathbf y})\|_{\ell^2}
	\le  \|\nabla^2 G\|_{L^\infty(\wh \Omega)} \|\wh{\mathbf x}-\wh{\mathbf y}\|_{\ell^2}$.
	Assumption~\eqref{eq:ass:nabla} implies
	 $\|G(\wh{\mathbf x})-G(\wh{\mathbf y})\|_{\ell^2}
		 \eqsim H_k \|\wh{\mathbf x}-\wh{\mathbf y}\|_{\ell^2}$.
	Using this and the substitution rule, and again~\eqref{eq:ass:nabla}, we have further
	\begin{align*}
			|\nabla \wh u|_{H^{1/2}(\wh \Gamma_{k,\ell})}^2
			&\le \|\nabla G_k\|_{L^\infty(\wh\Omega)}^2 \int_{\wh \Gamma_{k,\ell}}\int_{\wh \Gamma_{k,\ell}}
			\frac{\|\nabla u({\mathbf x})- \nabla u({\mathbf y})\|_{\ell^2}^2}{\|\wh{\mathbf x}-\wh{\mathbf y}\|_{\ell^2}^2} \, \mathrm d\wh{\mathbf x} \, \mathrm d \wh{\mathbf y} \\
			&\qquad+\|\nabla^2 G_k\|_{L^\infty(\wh\Omega)}^2 \int_{\wh \Gamma_{k,\ell}}\int_{\wh \Gamma_{k,\ell}}
			\|\nabla u({\mathbf y})\|_{\ell^2}^2 \, \mathrm d\wh{\mathbf x} \, \mathrm d \wh{\mathbf y}
			\\
			&\eqsim \|\nabla G_k\|_{L^\infty(\wh\Omega)}^2  |\nabla u|_{H^{1/2}(\Gamma_{k,\ell})}^2
			+\|\nabla^2 G_k\|_{L^\infty(\wh\Omega)}^2 H_k^{-1}
			\|\nabla u\|_{L^2(\Gamma_{k,\ell})}^2 .
	\end{align*}
	Assumption~\eqref{eq:ass:nabla} gives then the desired result.
\end{proof}

Finally, we obtain the estimate for the norm of the extension operator.

\begin{theorem}\label{thrm:extension}
	Let $\Omega_k$ be one of the patches and $\Gamma_{k,\ell}$ be one of its sides.
	The statements
	\begin{equation}\nonumber
		\begin{aligned}
		&(E^{(k,\ell)} u)|_{\Gamma_{k,\ell}} = u|_{\Gamma_{k,\ell}},
		\quad
		\nabla (E^{(k,\ell)} u)|_{\Gamma_{k,\ell}} = \nabla u|_{\Gamma_{k,\ell}},\\
		&(E^{(k,\ell)} u)|_{\partial\Omega_k\setminus\Gamma_{k,\ell}} = 0,
		\quad
		\nabla (E^{(k,\ell)} u)|_{\partial\Omega_k\setminus\Gamma_{k,\ell}} = 0
		\end{aligned}
	\end{equation}
	and the estimate
	\begin{align*}
		|E^{(k,\ell)} u|_{H^2(\Omega_k)}^2
		\lesssim
		p^3
		|\nabla u|_{H^{1/2}(\Gamma_{k,\ell})}^2
		+
		 p^3 \left(1+\log p + \log \frac{H_k}{h_k}\right)  | \nabla u |_{L_0^{\infty}(\Gamma_{k,\ell})}^2
	\end{align*}
	hold for all $u\in V_\Delta^{(k)}$.
\end{theorem}
\begin{proof}
	Let $C_\Lambda:=1+\log p + \log H_k/h_k = 1+\log p + \log 1/\wh h_k $.
	Let $u \in V^{(k)}_{\Delta}$ and let $\wh u = u\circ G_k$. Using Lemma~\ref{lem:geoequiv1},
	using Theorem~\ref{thrm:ealpha bound}
	and using Lemmas~\ref{lem:geoequiv1} and~\ref{lem:geoequiv2}
	and $\|w\|_{L^2(\Gamma_{k,\ell})} \lesssim H_k \|w\|_{L^\infty(\Gamma_{k,\ell})}^2$,
	we have
	\begin{align*}
		|E^{(k,\ell)} u|_{H^2(\Omega_k)}^2
			& \lesssim H_k^{-2}
				\|\wh E  \, \wh u \|_{H^2(\wh\Omega)}^2
			 \lesssim H_k^{-2} p^3
				( |\nabla \wh u |_{H^{1/2}(\wh \Gamma)}^2
				+ C_\Lambda \|\nabla \wh u\|_{L^\infty(\wh\Gamma)}^2 ) \\
			& \lesssim p^3
				( |\nabla  u|_{H^{1/2}( \Gamma_{k,\ell})}^2
				+ H_k^{-1} \|\nabla  u\|_{L^{2}( \Gamma_{k,\ell})}^2
				+C_\Lambda \|\nabla u\|_{L^\infty(\Gamma_{k,\ell})}^2 ) \\
			& \lesssim p^3
				( |\nabla  u|_{H^{1/2}( \Gamma_{k,\ell})}^2
				+C_\Lambda \|\nabla u\|_{L^\infty(\Gamma_{k,\ell})}^2 ) .
	\end{align*}
	Since  $w:=\nabla u$ vanishes at both ends of $\Gamma_{k,\ell}$,
	we have
	$
			\|w\|_{L^\infty(\Gamma_{k,\ell})}
			\le 2
			|w|_{L_0^\infty(\Gamma_{k,\ell})},
	$
	which finishes the proof.
\end{proof}

\section{The condition number estimate}
\label{sec:5}

In this section, we give a condition number estimate for the proposed IETI-DP solver.
Using the notation $S_{\Gamma\Gamma}:=\mbox{diag } (S_{\Gamma\Gamma}^{(1)}, \dots, S_{\Gamma\Gamma}^{(K)})$ and $\widetilde S:= \mbox{diag }(S_{\Gamma\Gamma}, S_{\Pi\Pi})$ as well as
$B_{\Gamma} := (B_{\Gamma}^{(1)}, \dots, B_{\Gamma}^{(k)})$
and
$\widetilde B:= (B_\Gamma, B_\Pi)$, we obtain immediately from~\eqref{eq:3:Fdef}
and~\eqref{eq:3:Mdef} that
\[
	F = \widetilde B \widetilde S^{-1} \widetilde B^\top
	\quad\mbox{and}\quad
	M = \tfrac14 B_\Gamma S_{\Gamma\Gamma} B_\Gamma^\top.
\]
We first observe as follows.
\begin{lemma}\label{lem:BBtop}
		The identity $B_\Gamma B_\Gamma^\top = 2 I$  holds. % and $C^{(k)} B^{(k)\top} = 0$
\end{lemma}
\begin{proof}
	Each Lagrange multiplier realizes a constraint of the form~\eqref{eq:def:b}
	or~\eqref{eq:def:b2}. Thus, for each $\ell$, there are exactly two indices $i$ such that
	$[B_\Gamma]_{\ell,i} \ne 0$. Specifically, those non-zero coefficients are $1$ or $-1$.
	This immediately shows that $[B_\Gamma B_\Gamma^\top]_{\ell,\ell} = \sum_i [B_\Gamma]_{\ell,i}^2  = 2$.
	Moreover, for each degree of freedom, at most one Lagrange multiplier is active, i.e., for each $i$ there is at most one index $\ell$ such that $[B_\Gamma]_{\ell,i} \ne 0$. This immediately implies $[B_\Gamma B_\Gamma^\top]_{\ell,m}=0$ for $\ell\ne m$. This finishes the proof.
\end{proof}

Next, we discuss the relation between spline functions and associated coefficient vectors. Let
\[
	\widetilde V
	:=
	\left\{
		(v^{(1)},\dots,v^{(k)}) \in \prod_{k=1}^K V_k
		:
		C^{(k)} \ul v^{(k)} = R^{(k)} \ul w_\Pi \mbox{ for all $k$ and some } \ul w_\Pi
	\right\},
\]
where $\ul v^{(k)}$ is the coefficient vector associated to $v^{(k)}$. This definition of $\widetilde V$ represents the function space, where the values of the primal degrees of freedom agree between the patches; this definition coincides with the constraint from~\eqref{eq:3:contC}. Analogously, we define
\[
	\widetilde V_\Delta
	:=
	\left\{
		(v^{(1)},\dots,v^{(k)}) \in \widetilde V
		:
		C^{(k)} \ul v^{(k)} = 0 \mbox{ for all $k$ }
	\right\}.
\]
As introduced in Section~\ref{sec:3}, each coefficient vectors has the form $\ul v^{(k)} = (\ul v^{(k)\top}_I, \ul v^{(k)\top}_\Gamma, \ul v^{(k)\top}_\Pi)^\top \in \mathbb R^{N_k}$ with $N_k=N_{k,I}+N_{k,\Gamma}+N_{k,\Pi}$. We define the operator $Q_\Gamma: \widetilde V \to \mathbb R^{N_{1,\Gamma}+\cdots+N_{K,\Gamma}}$ via
\[
		Q_\Gamma v := ( \ul v^{(1)\top}_\Gamma, \dots,  \ul v^{(K)\top}_\Gamma)^\top.
\]
The operator $Q_\Pi : \widetilde V \to \mathbb R^{N_\Pi}$ assigns to each $v\in \widetilde V$ the vector $\ul w_\Pi$ such that $C^{(k)} \ul v^{(k)} = R^{(k)} \ul w_\Pi$ for all $k\in \{1,\dots,K\}$. Finally, let $\widetilde Q v:=((Q_\Gamma v)^\top, (Q_\Pi v)^\top)^\top$.

\begin{lemma}\label{lem:bbt}
	Let $u \in \widetilde V_\Delta$ and $w \in \widetilde V$ be such that
	$Q_\Gamma u
	 = B_\Gamma^\top \widetilde B
	\widetilde Q w$. Then
	\[
				\nabla u^{(k)} |_{\Gamma^{(k,\ell)}}
				=
				\nabla w^{(k)} |_{\Gamma^{(k,\ell)}}
				-
				\nabla w^{(k)} |_{\Gamma^{(k,\ell)}}
				\quad\mbox{and}\quad
				u^{(k)} |_{\Gamma^{(k,\ell)}}
				=
				w^{(k)} |_{\Gamma^{(k,\ell)}}
				-
				w^{(k)} |_{\Gamma^{(k,\ell)}}
	\]
	holds for all interfaces $\Gamma^{(k,\ell)}$.
\end{lemma}
\begin{proof}
	The observation that $B_\Gamma^\top \widetilde B$ models the jump is standard, see also Ref.~\refcite{SchneckenleitnerTakacs:2020}, where the details have been worked out for $C^0$ smoothness condition was considered. Since the conditions~\eqref{eq:3:smoothness1} and the conditions~\eqref{eq:3:smoothness2} refer to different sets of degrees of freedom, the results from Ref.~\refcite{SchneckenleitnerTakacs:2020} directly carry over to the setup of this problem and one obtains
	\begin{equation}\nonumber
			\wh u^{(k)}|_{\wh \Gamma} = \wh w^{(k)}|_{\wh\Gamma} - \wh w^{(\ell)}|_{\wh\Gamma},
			\qquad
			\partial_{\wh x} \wh u^{(k)}|_{\wh \Gamma} = \partial_{\wh x}  \wh w^{(k)}|_{\wh \Gamma} - \partial_{\wh x}  \wh w^{(\ell)}|_{\wh \Gamma},
	\end{equation}
	where both $\wh u^{(k)}$ and $\wh u^{(\ell)}$ are the pre-images of $u^{(k)}$ and $w^{(k)}$, respectively, under $G^{(k,\ell)}$ as defined in~\eqref{eq:gkl}. We immediately obtain
	\begin{equation}\nonumber
			\nabla \wh u^{(k)}|_{\wh \Gamma} = \nabla  \wh w^{(k)}|_{\wh \Gamma} - \nabla  \wh w^{(\ell)}|_{\wh \Gamma}
	\end{equation}
	and, since $G^{(k,\ell)}$ is $C^1$-smooth, the desired result.
\end{proof}

\begin{lemma}\label{lem:my trace}
	For each patch $\Omega_k$ and each of its sides $\Gamma_{k,\ell}$, the estimate
	\[
		|\nabla u|_{H^{1/2}(\Gamma_{k,\ell})}^2
		+
		|\nabla u|_{L^\infty_0(\Gamma_{k,\ell})}^2
		\lesssim
		(1+\log p+ \log \tfrac{H_k}{h_k})
		|u|_{H^2(\Omega_k)}^2
	\]
	holds for all $u\in H^2(\Omega_k)$.
\end{lemma}
\begin{proof}
	\eqref{eq:ass:trace} yields
	$
	|v|_{H^{1/2}(\Gamma_{k,\ell})}^2
	\lesssim
	|v|_{H^1(\Omega_k)}^2
	$
	for all $v\in H^1(\Omega_k)$.
	Using Lemma~4.14 in Ref.~\refcite{SchneckenleitnerTakacs:2020}
	and Lemma~\ref{lem:geoequiv1}, we also have
	$
	|v|_{L_0^{\infty}(\Gamma_{k,\ell})}^2
	\lesssim
		(1+\log p+ \log \tfrac{H_k}{h_k})
	|v|_{H^1(\Omega_k)}^2.
	$
	The desired estimate is obtained by applying these estimates to $v:=\nabla u$.
\end{proof}

\begin{lemma}\label{lem:omega}
	For each $w\in \widetilde V$, we have
	\[
		\|B_{\Gamma}^\top \widetilde B \widetilde Q w\|_{S_{\Gamma\Gamma}}^2
		\lesssim p^3\left(1+\log p + \max_{k=1,\dots,K}  \log \frac{H_k}{h_k}\right)^2
		\| \widetilde Q w \|_{\widetilde S}^2.
	\]
\end{lemma}
\begin{proof}
	Define $C_\Lambda:=1+\log p + \max_{k=1,\dots,K}  \log \frac{H_k}{h_k}$.
	Let $w \in \widetilde V$ be arbitrary but fixed and
	$\ul w := \widetilde Q w$.
	Moreover, let $u\in \widetilde V_\Delta$ be arbitrary but fixed
	such that $\ul u := Q_\Delta u = B_{\Gamma}^\top \widetilde B \ul w$.
	The Schur complement $S_{\Gamma\Gamma}^{(k)}$ models the biharmonic extension, i.e., we have
	$\|\ul u \|_{S_{\Gamma\Gamma}}^2
	 = \sum_{k=1}^K\|\ul u_\Gamma^{(k)} \|_{S_{\Gamma\Gamma}^{(k)}}^2$ with
	\begin{align*}\nonumber
		\|\ul u_\Gamma^{(k)} \|_{S_{\Gamma\Gamma}^{(k)}}^2
			& =
				\inf_{v \in V^{(k)}_\Delta : v|_{\partial\Omega_k} = u^{(k)}|_{\partial\Omega_k},\partial_n v|_{\partial\Omega_k} = \partial_n u^{(k)}|_{\partial\Omega_k}}
				| v |_{H^{2}(\Omega_k)}^2 \\
			& =
				\inf_{v \in V^{(k)}_\Delta : \nabla v|_{\partial\Omega_k} = \nabla u^{(k)}|_{\partial\Omega_k}}
				| v |_{H^{2}(\Omega_k)}^2
	\end{align*}
	and consequently
	\begin{equation}\label{eq:lm52:schur1v}
			\|\ul u \|_{S_{\Gamma\Gamma}}^2
			=
			\sum_{k=1}^K\;
				\inf_{v \in V^{(k)}_\Delta : \nabla v|_{\partial\Omega_k} = \nabla u^{(k)}|_{\partial\Omega_k}}
				| v |_{H^{2}(\Omega_k)}^2.
	\end{equation}
	Analogously, we have
	\begin{equation}\label{eq:lm52:schur1w}
			\|\ul w\|_{\widetilde S}^2
			=
			\sum_{k=1}^K\;
				\inf_{v \in V^{(k)} : \nabla v|_{\partial\Omega_k} = \nabla w^{(k)}|_{\partial\Omega_k}}
				| v |_{H^{2}(\Omega_k)}^2.
	\end{equation}
	Using~\eqref{eq:lm52:schur1v} and Theorem~\ref{thrm:extension}, we obtain
	\begin{equation}\nonumber
	\begin{aligned}
		\|\ul u \|_{S_{\Gamma\Gamma}}^2
		& \le \sum_{k=1}^K \big|
					 \sum_{\ell \in \mathcal N(k)} E^{(k,\ell)} u^{(k)} \big|_{H^2(\Omega_k)}^2
			\lesssim \sum_{k=1}^K \sum_{\ell \in \mathcal N(k)}
					|E^{(k,\ell)} u^{(k)} |_{H^2(\Omega_k)}^2
		\\
		& \lesssim p^3 C_\Lambda \sum_{k=1}^K \sum_{\ell\in\mathcal N(k)} \big(
				| \nabla u^{(k)} |_{H^{1/2}(\Gamma_{k,\ell})}^2
				+
				 | \nabla u^{(k)} |_{L_0^\infty(\Gamma_{k,\ell})}^2
			\big),
	\end{aligned}
	\end{equation}
	where $\mathcal N(k):=\{\ell : \Omega_k \mbox{ and }\Omega_\ell \mbox{ share edge}\}$.
	Using Lemma~\ref{lem:bbt} and the triangle inequality and $\ell \in \mathcal N(k) \Leftrightarrow k \in \mathcal N (\ell)$, we further have
	\begin{equation}\nonumber
	\begin{aligned}
		\|\ul u \|_{S_{\Gamma\Gamma}}^2
		& \lesssim p^3 C_\Lambda \sum_{k=1}^K \sum_{\ell\in\mathcal N(k)} \big(
				| \nabla w^{(k)} |_{H^{1/2}(\Gamma_{k,\ell})}^2
				+
				 | \nabla w^{(k)} |_{L_0^\infty(\Gamma_{k,\ell})}^2
			\big).
	\end{aligned}
	\end{equation}
	Using Lemma~\ref{lem:my trace}, we further have
	\begin{equation}\label{eq:lm52:schur2}
	\begin{aligned}
		\|\ul u \|_{S_{\Gamma\Gamma}}^2
		\lesssim
		p^3C_\Lambda^2
		\sum_{k=1}^K
			\inf_{v \in V^{(k)} : \nabla v|_{\partial\Omega_k} = \nabla w^{(k)}|_{\partial\Omega_k}}
				| v |_{H^{2}(\Omega_k)}^2
		=p^3  C_\Lambda^2 \|\ul w \|_{\widetilde S}^2.
	\end{aligned}
	\end{equation}
	This finishes the proof.
\end{proof}

Now, we can prove the condition number estimate.
\begin{theorem}\label{thrm:main}
		There is a constant $C$ that only depends on
		$C_G$ from~\eqref{eq:ass:nabla}, $C_T$ from~\eqref{eq:ass:trace}, $C_Q$ from~\eqref{eq:ass:quasi uniform}, and the diameter of $\Omega$
		such that the estimate
		\[
				\kappa(M F)\le C p^3 \left(1+\log p + \max_{k=1,\dots,K}  \log \frac{H_k}{h_k}\right)^2
		 \]
		holds.
\end{theorem}
\begin{proof}
	Let $C_\Lambda :=1+\log p + \max_{k=1,\dots,K}  \log \frac{H_k}{h_k}$.
	We have using Lemma~\ref{lem:omega} that
	\begin{align*}
			\ul \lambda^\top F \ul \lambda
			& = \sup_{\ul {\widetilde w}} \frac{ |\ul {\widetilde w}^\top \widetilde B^\top \ul \lambda|^2 }{ \|\ul {\widetilde w} \|_{\widetilde S}^2}
			 \lesssim p^3 C_\Lambda^2 \sup_{\ul {\widetilde w}} \frac{ |\ul {\widetilde w}^\top \widetilde B^\top \ul \lambda|^2 }{ \|B_\Gamma ^\top \widetilde B \ul {\widetilde w} \|_{S_{\Gamma\Gamma}}^2}
			\\
			& \le p^3 C_\Lambda^2 \sup_{\ul \mu} \frac{ |\ul \mu^\top \ul \lambda|^2 }{ \|B_\Gamma ^\top \ul \mu \|_{S_{\Gamma\Gamma}}^2}
			= p^3 C_\Lambda^2\; \ul \lambda^\top (B_\Gamma  S_{\Gamma\Gamma} B_\Gamma ^\top)^{-1} \ul \lambda
			 = \frac{p^3 C_\Lambda^2}4 \ul \lambda^\top M^{-1} \ul \lambda
	\end{align*}
	for all vectors $\ul \lambda$.
	Next, we observe that Lemma~\ref{lem:BBtop} implies that $B_\Gamma B_\Gamma ^\top$ has full rank and that the rank must be equal to the dimensionality of the identity matrix, i.e., the matrix $B_\Gamma $ has as least as many columns as rows, i.e., that -- interpreted as a linear operator -- $B_\Gamma $ is surjective. Using this and Lemma~\ref{lem:BBtop} and $B_\Gamma  S_{\Gamma\Gamma}^{-1} B_\Gamma^\top  \le
		B_\Gamma  S_{\Gamma\Gamma}^{-1} B_\Gamma^\top
		+B_\Pi  S_{\Pi\Pi}^{-1} B_\Pi^\top
		= \widetilde B  \widetilde S^{-1} \widetilde B^\top =F$, we see
	\begin{align*}
			\ul\lambda^\top M^{-1} \ul \lambda
			& =4 \sup_{\ul w_\Gamma} \frac{ |\ul w_\Gamma^\top B_\Gamma ^\top \ul \lambda|^2 }{ \|B_\Gamma ^\top B_\Gamma  \ul w_\Gamma \|_{S_{\Gamma\Gamma}}^2}
			 \le \sup_{\ul w_\Gamma} \frac{ |\ul w_\Gamma^\top B_\Gamma^\top \ul \lambda|^2 }{ \|\ul w_\Gamma \|_{S_{\Gamma\Gamma}}^2}
			= \ul\lambda^\top B_\Gamma  S_{\Gamma\Gamma}^{-1} B_\Gamma^\top \ul \lambda
			\le \ul\lambda^\top F \ul \lambda
	\end{align*}
	for all vectors $\ul \lambda$.
	Since we have estimated the relative eigenvalues of $F$ and $M^{-1}$, we immediately obtain the desired result.
\end{proof}

This result shows that the condition number of the preconditioned system behaves --
analogous to the results obtained in Ref.~\refcite{SchneckenleitnerTakacs:2020} for the Poisson problem -- that the condition number bound follows a polylogarithmic growth with respect to the grid size. Concerning the growth with respect to the spline degree, in  Ref.~\refcite{SchneckenleitnerTakacs:2020}, the authors were able to show that the growth is like $p (\log p)^2$ for the Poisson problem, while we obtain  $p^3 (\log p)^2$ for the biharmonic problem. In both cases, the numerical experiments show that the growth in $p$ is actually not as severe as indicated by the convergence theory.
The main contribution of this paper however is the construction of a IETI-DP solver that does not require deluxe scaling, but uses a standard scaled Dirichlet preconditioner instead.

\section{Numerical results}
\label{sec:6}

In this section, we illustrate our theoretical findings with the results from numerical experiments. All numerical experiments are realized with the C++ library G+Smo, see Ref.~\refcite{gismoweb}.

\subsection{B-spline quarter annulus}

For the first experiment, we choose a B-spline approximation of a quarter annulus. The geometry is parameterized with B-splines of degree $2$ and $\Xi_1=\Xi_2=(0,0,0,1/2,1,1,1)$. The geometry and the control mesh are shown in Figure~\ref{fig:annulussolution}. This single-patch geometry is subdivided into $4\times 4$ or $8 \times 8$ patches by uniformly splitting the domain. Since the overall geometry is based on a single-patch geometry, this construction guarantees a matching $C^1$ conforming discretization of the multi-patch geometry.

\begin{figure}[h]
\begin{center}
\includegraphics[height=2.8cm]{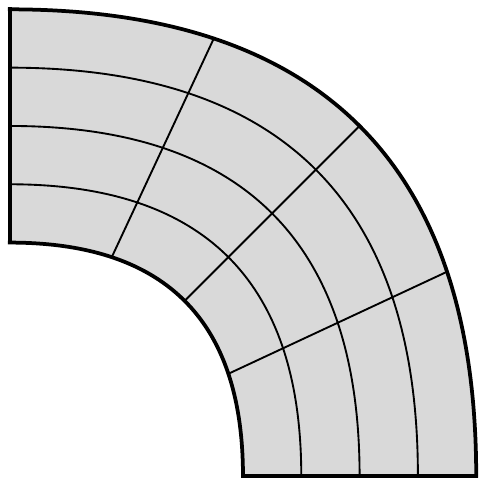}\quad
\includegraphics[height=2.8cm]{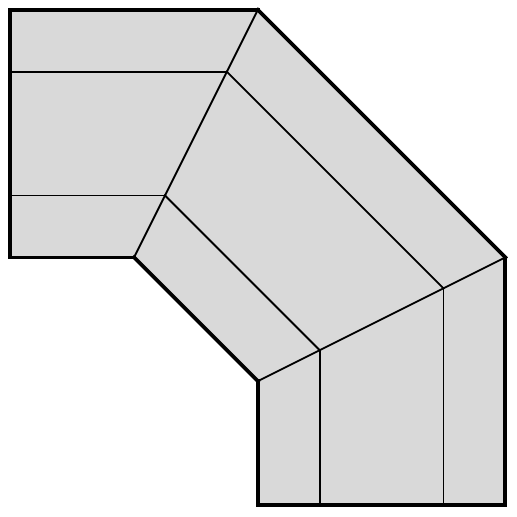}\quad
\includegraphics[height=2.8cm]{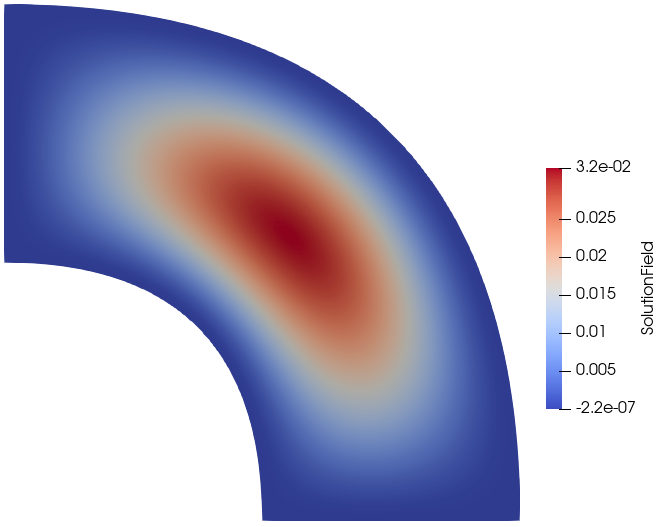}
\caption{Annulus domain, corresponding control mesh, and solution field as obtained with $16$ patches, $r=5$ and $p=3$}
\label{fig:annulussolution}
\end{center}
\end{figure}

\noindent
On this domain, we solve the first biharmonic problem
\[
	\Delta^2 u(x,y) = \tfrac 18 \pi^4 \sin(\pi x/2) \sin(\pi y/2)
	\;\;\mbox{for}\;\; (x,y) \in \Omega,
	\quad
	u=\partial_n u=0
	\;\;\mbox{on}\;\;\partial\Omega.
\]
We discretize the problem using grids that are (on the parameter domain) uniform. The refinement level $r=0$ corresponds to a grid without inner knots, all further grids are obtained by uniform dyadic refinement, so it has $2^r-1$ inner knots per direction (and $\wh h=2^{-r}$). For the discretization, we use splines of degree $p \in \{2,3,4,5,6\}$ and with  maximum smoothness. We set up the IETI-DP system as outlined in this paper. The resulting preconditioned system $F \ul \lambda = \ul b$ is solved with a conjugate gradient (cg) solver, preconditioned with the proposed preconditioner $M$. We stop the iteration if the residual was reduced by a factor of $10^{-6}$, compared to the initial residual. We give the numbers of iterations required to obtain the stopping criterion (iteration count) and the estimates for the condition number of $MF$ as obtained by the cg solver in Tables~\ref{tab:Annulus2std} and \ref{tab:Annulus3std}.
The solution for $p=3$ and $r=5$ over $16$ patches is depicted in Figure~\ref{fig:annulussolution}.

\begin{table}[t]
\scriptsize
	\newcolumntype{L}[1]{>{\raggedleft\arraybackslash\hspace{-1em}}m{#1}}
	\centering
	\renewcommand{\arraystretch}{1.25}
\begin{tabular}{l|L{2em}L{1em}|L{2em}L{1em}|L{2em}L{1em}|L{2em}L{1em}|L{2em}L{1em}}
		\toprule
		& \multicolumn{2}{c|}{$p=2$}
		& \multicolumn{2}{c|}{$p=3$}
		& \multicolumn{2}{c|}{$p=4$}
		& \multicolumn{2}{c|}{$p=5$}
		& \multicolumn{2}{c}{$p=6$}
\\
		$r$
		& $\kappa$ & it
		& $\kappa$ & it
		& $\kappa$ & it
		& $\kappa$ & it
		& $\kappa$ & it \\
		\midrule
3 & 16.61 & 33 & 16.47 & 34 & 19.44 & 36 & 24.86 & 40 &  27.71 & 45\\
4 & 27.50 & 40 & 23.58 & 39 & 30.19 & 42 & 36.64 & 46 &  42.22 & 50\\
5 & 37.55 & 47 & 35.82 & 46 & 43.16 & 50 & 51.90 & 53 &  60.31 & 59\\
6 & 50.77 & 53 & 51.39 & 54 & 62.91 & 57 & 70.73 & 61 &  79.59 & 66\\
		\bottomrule
	\end{tabular}
	\captionof{table}{Iteration counts (it) and condition $\kappa$; Annulus composed of $16$ patches
		\label{tab:Annulus2std}}
\end{table}

\begin{table}[t]
\scriptsize
	\newcolumntype{L}[1]{>{\raggedleft\arraybackslash\hspace{-1em}}m{#1}}
	\centering
	\renewcommand{\arraystretch}{1.25}
\begin{tabular}{l|L{2em}L{1em}|L{2em}L{1em}|L{2em}L{1em}|L{2em}L{1em}|L{2em}L{1em}}
		\toprule
		& \multicolumn{2}{c|}{$p=2$}
		& \multicolumn{2}{c|}{$p=3$}
		& \multicolumn{2}{c|}{$p=4$}
		& \multicolumn{2}{c|}{$p=5$}
		& \multicolumn{2}{c}{$p=6$}
\\
		$r$
		& $\kappa$ & it
		& $\kappa$ & it
		& $\kappa$ & it
		& $\kappa$ & it
		& $\kappa$ & it \\
		\midrule
3 & 22.73 & 40 & 26.58 & 44 & 27.98 & 46 & 33.92 & 52 & 38.61 & 57\\
4 & 33.02 & 49 & 33.20 & 51 & 39.08 & 55 & 47.63 & 62 & 54.66 & 66\\
5 & 42.48 & 56 & 46.52 & 60 & 58.65 & 67 & 66.78 & 72 & 73.51 & 76\\
6 & 56.44 & 66 & 65.99 & 72 & 80.88 & 79 & 94.07 & 84 & 94.19 & 87\\
		\bottomrule
	\end{tabular}
	\captionof{table}{Iteration counts (it) and condition $\kappa$; Annulus composed of $64$ patches
		\label{tab:Annulus3std}}
\end{table}

We observe that the condition number grows approximately linear with the refinement level $r$ (and thus with $\max_k \log H_k/h_h$). Concerning the dependence on the spline degree, it seems that the growth is smaller than quadratic in $p$. Both observations are consistent with the theory.
For the finer grid levels, it seems that the condition numbers for $p=2$ are larger than those for $p=3$.

\subsection{D-shaped lamella with hole}\label{sec:6:2}

\begin{figure}[b]
\begin{center}
\includegraphics[height=3cm]{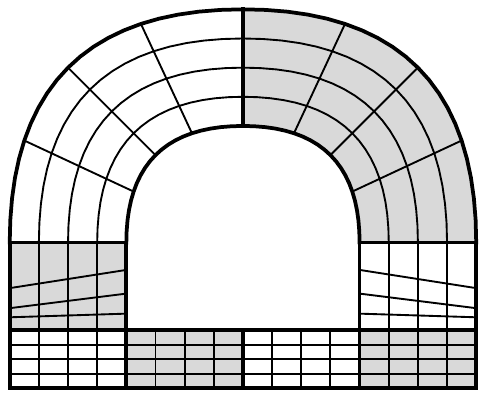}\quad
\includegraphics[height=3cm]{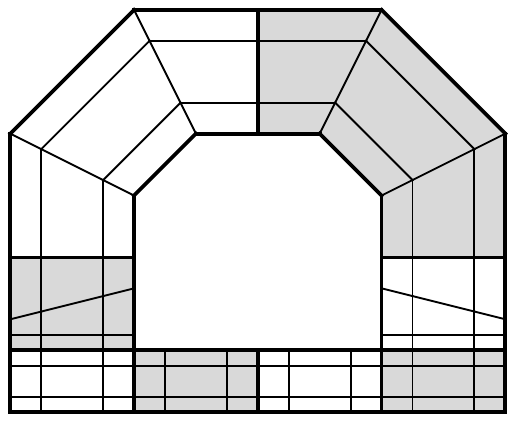}\quad
\includegraphics[height=3cm]{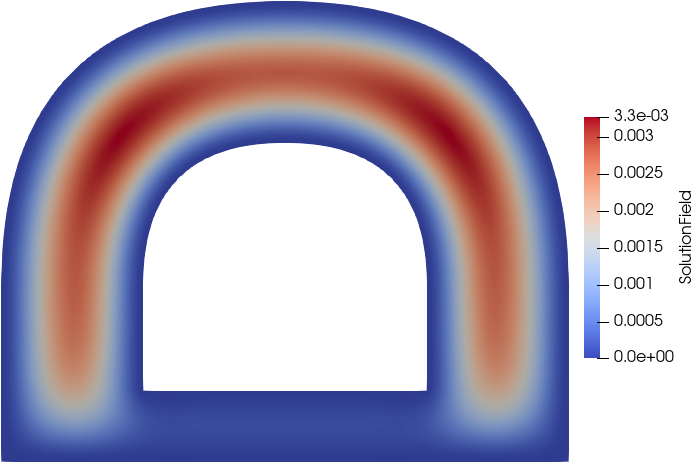}
\caption{Lamella domain, corresponding control mesh, and solution field as obtained with $32$ patches, $r=5$ and $p=3$}
\label{fig:lamella:1}
\end{center}
\end{figure}

\begin{table}[t]
\scriptsize
	\newcolumntype{L}[1]{>{\raggedleft\arraybackslash\hspace{-1em}}m{#1}}
	\centering
	\renewcommand{\arraystretch}{1.25}
\begin{tabular}{l|L{2em}L{1em}|L{2em}L{1em}|L{2em}L{1em}|L{2em}L{1em}|L{2em}L{1em}}
		\toprule
		& \multicolumn{2}{c|}{$p=2$}
		& \multicolumn{2}{c|}{$p=3$}
		& \multicolumn{2}{c|}{$p=4$}
		& \multicolumn{2}{c|}{$p=5$}
		& \multicolumn{2}{c}{$p=6$}
\\
		$r$
		& $\kappa$ & it
		& $\kappa$ & it
		& $\kappa$ & it
		& $\kappa$ & it
		& $\kappa$ & it \\
		\midrule
3 & 20.73 & 38 & 14.24 & 33 & 15.87 & 36 & 18.04 & 40 & 23.72 & 46 \\
4 & 38.15 & 47 & 27.43 & 41 & 24.25 & 43 & 26.33 & 46 & 28.52 & 52 \\
5 & 66.13 & 56 & 44.16 & 49 & 36.44 & 51 & 38.35 & 54 & 42.93 & 58 \\
6 & 99.57 & 65 & 63.91 & 58 & 49.75 & 59 & 52.50 & 62 & 56.53 & 66 \\
		\bottomrule
	\end{tabular}
	\captionof{table}{Iteration counts (it) and condition $\kappa$; Lamella domain composed of $32$ patches
		\label{tab:Lamella1std}}
\end{table}

In this second experiment, we demonstrate that the proposed approach does not only work on geometries that are obtained by splitting single-patch geometries. As initial geometry, we consider a D-shaped lamella with a hole. The initial geometry is represented by $8$ patches. Each patch is parameterized with B-splines of degree $2$ and with knot vectors $(0,0,0,1/2,1,1,1)$ as above. The geometry and control mesh are depicted in Figure~\ref{fig:lamella:1}.
The rectangular patches adjacent to the quarter annulus domains are parameterized in a non-standard way in order to obtain a $C^1$-fully matching discretization.
Then, each patch is uniformly split into 4 patches, so the geometry is represented using $32$ patches.
On this domain, we solve the model problem
\[
	\Delta^2 u = 1
	\;\;\mbox{on}\;\;  \Omega,
	\quad
	u=\partial_n u=0
	\;\;\mbox{on}\;\;\partial\Omega.
\]
Again, we have discretized the problem with splines of some degree $p$, maximum smoothness over a uniform grid with $\wh h=2^{-r}$ as obtained by $r$ dyadic refinement steps. The setup of the cg solver is as in the last subsection. The iteration counts and the condition number estimates for $MF$ as obtained by the cg solver are given in Table~\ref{tab:Lamella1std}.
We observe that the growth of the condition number in the refinement level $r$ (and thus in $\max_k \log H_k/h_h$) is approximately linear. Concerning the dependence on the spline degree, it seems that the growth is smaller than quadratic in $p$. This is again consistent with the theory.

\subsection{A modified preconditioner}

For comparization, we consider a modified preconditioner that we motivate for a simple two-patch domain only; assuming that $\Omega_1 := (-1,0)\times(0,1)$ and $\Omega_2:=(0,1)\times (0,1)$ are discretized with the same function, we have
\[
		F = B_\Gamma^{(1)} S_{\Gamma\Gamma}^{-1} B_\Gamma^{(1)\top}
				+ B_\Gamma^{(2)} S_{\Gamma\Gamma}^{-1} B_\Gamma^{(2)\top},
\]
where $S_{\Gamma\Gamma}=S_{\Gamma\Gamma}^{(1)}=S_{\Gamma\Gamma}^{(2)}$ is the same matrix due to symmetry. (There is no term $B_\Pi S_{\Pi\Pi}^{-1}  B_\Pi^\top$ since there are no primal degrees of freedom.) Next, we split the degrees of freedom of group ``$\Gamma$'' further into degrees of freedom with non-vanishing function value (group ``$V$'', depicted with blue square dots in Figure~\ref{fig:schema}) and degrees of freedom with non-vanishing normal derivative (group ``$D$'', depicted with red triangular dots). Based on this splitting and assuming a corresponding order of the degrees of freedom, have
\[
		S_{\Gamma\Gamma} = \begin{pmatrix} S_{VV} & S_{VD} \\ S_{DV} & S_{DD} \end{pmatrix},
		\quad
		B_{\Gamma}^{(1)} = \begin{pmatrix} I & 0 \\ 0 & I \end{pmatrix},
		\quad
		B_{\Gamma}^{(2)} = \begin{pmatrix} -I & 0 \\ 0 & I \end{pmatrix}.
\]
The matrices $B_{\Gamma}^{(1)}$ and $B_{\Gamma}^{(2)}$ have different signs in their first block since the constraint~\eqref{eq:def:b} for the function values guarantees equality. In the second block, the sign is always positive since the normal derivatives should have different signs, which is guaranteed by constraint~\eqref{eq:def:b2}.
Following this, we obtain
\begin{align*}
		F &= 2 \begin{pmatrix}
			S_{VV}-S_{VD}S_{DD}^{-1}S_{DV} \\ &S_{DD}-S_{DV}S_{VV}^{-1}S_{VD}
		\end{pmatrix}^{-1}
		\\&= \sum_k
		B_{\Gamma}^{(k)}
		\begin{pmatrix}
			S_{VV}^{(k)}-S_{VD}^{(k)}S_{DD}^{(k)-1}S_{DV}^{(k)} \\ &S_{DD}^{(k)}-S_{DV}^{(k)}S_{VV}^{(k)-1}S_{VD}^{(k)}
		\end{pmatrix}^{-1}
		B_{\Gamma}^{(k)\top}.
\end{align*}
This motivates the choice
\[
		M_{mod}
		:=
		\sum_k
		B_{\Gamma}^{(k)}
		\begin{pmatrix}
			S_{VV}^{(k)}-S_{VD}^{(k)}S_{DD}^{(k)-1}S_{DV}^{(k)} \\ &S_{DD}^{(k)}-S_{DV}^{(k)}S_{VV}^{(k)-1}S_{VD}^{(k)}
		\end{pmatrix}
		B_{\Gamma}^{(k)\top}.
\]
This modified preconditioner is not only applicable for the two-patch case, but also for the multi-patch case. Again, the Schur complements can be represented using the matrix $A^{(k)}$. In order to realize the overall preconditioner, two patch-local problems have to be solved, one for the group ``$V$'' and one for the group ``$D$''. This makes the preconditioner around twice as expensive as the proposed preconditioner. The results from the numerical experiments as presented in Tables~\ref{tab:Annulus2mod}, \ref{tab:Annulus3mod} and~\ref{tab:Lamella1mod} show that using this preconditioner leads to the same qualitative behavior as using the proposed preconditioner, generally with reduced iteration counts. Although the iteration counts are smaller, this does not compensate for the higher costs per iteration.

\begin{table}[h]
\scriptsize
	\newcolumntype{L}[1]{>{\raggedleft\arraybackslash\hspace{-1em}}m{#1}}
	\centering
	\renewcommand{\arraystretch}{1.25}
\begin{tabular}{l|L{2em}L{1em}|L{2em}L{1em}|L{2em}L{1em}|L{2em}L{1em}|L{2em}L{1em}}
		\toprule
		& \multicolumn{2}{c|}{$p=2$}
		& \multicolumn{2}{c|}{$p=3$}
		& \multicolumn{2}{c|}{$p=4$}
		& \multicolumn{2}{c|}{$p=5$}
		& \multicolumn{2}{c}{$p=6$}
\\
		$r$
		& $\kappa$ & it
		& $\kappa$ & it
		& $\kappa$ & it
		& $\kappa$ & it
		& $\kappa$ & it \\
		\midrule
3  &  6.04  &  18  &  6.47  &  19  &  7.50  &  20  &  9.22  &  23  &  17.83  &  27\\
4  &  13.72  &  25  &  11.14  &  24  &  11.65  &  24  &  13.83  &  26  &  20.42  &  31\\
5  &  21.13  &  29  &  15.95  &  28  &  17.03  &  28  &  19.74  &  30  &  25.10  &  35\\
6  &  29.27  &  34  &  21.26  &  32  &  23.65  &  31  &  26.90  &  33  &  30.79  &  38\\
		\bottomrule
	\end{tabular}
	\captionof{table}{Iteration counts (it) and condition $\kappa$; Annulus composed of $16$ patches using modified Dirichlet preconditioner
		\label{tab:Annulus2mod}}
\end{table}

\begin{table}[h]
\scriptsize
	\newcolumntype{L}[1]{>{\raggedleft\arraybackslash\hspace{-1em}}m{#1}}
	\centering
	\renewcommand{\arraystretch}{1.25}
\begin{tabular}{l|L{2em}L{1em}|L{2em}L{1em}|L{2em}L{1em}|L{2em}L{1em}|L{2em}L{1em}}
		\toprule
		& \multicolumn{2}{c|}{$p=2$}
		& \multicolumn{2}{c|}{$p=3$}
		& \multicolumn{2}{c|}{$p=4$}
		& \multicolumn{2}{c|}{$p=5$}
		& \multicolumn{2}{c}{$p=6$}
\\
		$r$
		& $\kappa$ & it
		& $\kappa$ & it
		& $\kappa$ & it
		& $\kappa$ & it
		& $\kappa$ & it \\
		\midrule
3 &  8.26 & 22 &  7.97 & 24 & 10.37 & 27 & 12.85 & 31 & 21.98 & 39\\
4 & 17.21 & 32 & 14.15 & 32 & 16.73 & 34 & 20.18 & 36 & 25.74 & 43\\
5 & 26.79 & 40 & 20.04 & 39 & 25.80 & 41 & 30.06 & 44 & 33.51 & 49\\
6 & 38.63 & 49 & 29.54 & 47 & 37.00 & 47 & 43.82 & 51 & 48.35 & 56\\
		\bottomrule
	\end{tabular}
	\captionof{table}{Iteration counts (it) and condition $\kappa$; Annulus composed of $64$ patches using modified Dirichlet preconditioner
		\label{tab:Annulus3mod}}
\end{table}

\begin{table}[h]
\scriptsize
	\newcolumntype{L}[1]{>{\raggedleft\arraybackslash\hspace{-1em}}m{#1}}
	\centering
	\renewcommand{\arraystretch}{1.25}
\begin{tabular}{l|L{2em}L{1em}|L{2em}L{1em}|L{2em}L{1em}|L{2em}L{1em}|L{2em}L{1em}}
		\toprule
		& \multicolumn{2}{c|}{$p=2$}
		& \multicolumn{2}{c|}{$p=3$}
		& \multicolumn{2}{c|}{$p=4$}
		& \multicolumn{2}{c|}{$p=5$}
		& \multicolumn{2}{c}{$p=6$}
\\
		$r$
		& $\kappa$ & it
		& $\kappa$ & it
		& $\kappa$ & it
		& $\kappa$ & it
		& $\kappa$ & it \\
		\midrule
3  &    6.86  &  19  &   6.64  &  19  &   6.24  &  19  &   8.02  &  24  &  16.67  &  31\\
4  &   21.81  &  27  &  17.61  &  25  &  14.01  &  24  &  11.88  &  28  &  19.62  &  34\\
5  &   42.42  &  34  &  32.21  &  30  &  24.10  &  28  &  19.05  &  31  &  24.01  &  38\\
6  &   68.66  &  41  &  49.60  &  36  &  36.66  &  33  &  28.05  &  35  &  29.83  &  42\\
		\bottomrule
	\end{tabular}
	\captionof{table}{Iteration counts (it) and condition $\kappa$; Lamella domain composed of $32$ patches using modified Dirichlet preconditioner
		\label{tab:Lamella1mod}}
\end{table}

\section{Conclusions}
\label{sec:7}

We have presented a IETI-DP solver for the first biharmonic problem.
The proposed discretization uses a $C^1$ smooth spline discretization across the patches, so it is $H^2$-conforming. For the IETI-DP solver, we have proposed a patch wise preconditioner that is a variant of the scaled Dirichlet preconditioner, which is the standard for second order problems. The key idea that allowed the construction and the analysis of this approach was the construction of an appropriate basis transformation. For the proposed approach, we have shown (Theorem~\ref{thrm:main}) that the condition number of the preconditioned system follows the standard behavior in the grid size, i.e., like $(1+\max_k \log H_k/h_k)^2$. We have also given an explicit estimate with respect to the spline degree, which does not seem to be sharp. The numerical experiments illustrate our theoretical findings.

\section*{Acknowledgments}
The author wants to thank Rainer Schneckenleitner with whom the author discussed the setup of IETI-DP solvers for biharmonic problems a few years ago.
This research was funded in whole or in part by the Austrian Science Fund (FWF): 10.55776/P33956.

\bibliography{literature.bib}

\end{document}